\documentclass[11pt
]{amsart}
\usepackage{amsmath, amssymb}

\usepackage{esint}

\usepackage{graphicx}

\usepackage{hyperref}
\hypersetup{bookmarksdepth=3}

    \normalbaselines                          %

\numberwithin{equation}{section}

\newtheorem{thm}{Theorem}[section]
\newtheorem{prop}[thm]{Proposition}

\newtheorem{lm}[thm]{Lemma}

\theoremstyle{remark}
\newtheorem{rem}[thm]{Remark}

\newtheorem*{rem*}{Remark}

\theoremstyle{definition}
\newtheorem{df}[thm]{Definition}
\newtheorem*{df*}{Definition}

\newtheorem*{ex*}{Example}

\newenvironment{entry}
{\begin{list}{X}%
  {%
      \setlength{\labelwidth}{55pt}%
      \setlength{\leftmargin}{\labelwidth}
      \addtolength{\leftmargin}{\labelsep}%
      \setlength{\itemsep}{.4pc}
   }%
}%
{\end{list}}

\newcounter{vremennyj}

\newcommand\cond[1]{\setcounter{vremennyj}{\theenumi}\setcounter{enumi}{#1}\labelenumi\setcounter{enumi}{\thevremennyj}}

\newcommand{\s}{\sigma}

\newcommand{\rk}{\operatorname{rk}}

\newcommand{\f}{\varphi}

\newcommand{\e}{\varepsilon}

\newcommand{\C}{\mathbb{C}}
\newcommand{\R}{\mathbb{R}}
\newcommand{\Z}{\mathbb{Z}}
\newcommand{\N}{\mathbb{N}}
\newcommand{\E}{\mathbb{E}}

\newcommand{\X}{\mathbb{X}}

\newcommand{\cE}{\mathcal{E}}
\newcommand{\cJ}{\mathcal{J}}

\newcommand{\1}{\mathbf{1}}

\newcommand{\BMO}{\ensuremath\textup{BMO}}
\newcommand{\BMOs}{\ensuremath\operatorname{\mathcal{BMO}}}

\newcommand{\La}{\langle }
\newcommand{\Ra}{\rangle }

\newcommand{\om}{\omega}

\newcommand{\cH}{\mathcal{H}}

\newcommand{\cP}{\mathcal{P}}
\newcommand{\cC}{\mathcal{C}}

\newcommand{\cL}{\mathcal{L}}

\newcommand{\cG}{\mathcal{G}}

\newcommand{\cD}{\mathcal{D}}

\newcommand{\fb}{\mathbf{\dot f}}
\newcommand{\gb}{\mathbf{\dot g}}

\newcommand{\fdot}{\,\cdot\,}

\newcommand{\chl}{\operatorname{child}}
\newcommand{\chld}{\operatorname{child}}

\newcommand\spn{\operatorname{span}}

\newcommand{\rf}{\operatorname{r}}
\newcommand{\Av}{\operatorname{A\!v}}

\newcommand{\Aio}{\ensuremath\mathfrak A_{-\infty}^{0}}
\newcommand{\Aif}{\ensuremath\mathfrak A_{-\infty}^{0, \textup{fin}}}

\renewcommand{\labelenumi}{(\roman{enumi})}

\newcommand{\ci}[1]{_{ {}_{\scriptstyle #1}}}
\newcommand{\ti}[1]{_{\scriptstyle \text{\rm #1}}}

\newcommand{\crc}[1]{\overset{\raisebox{-.3ex}{$\scriptscriptstyle\circ$}}{#1}}

\newcommand{\wt}{\widetilde}
\newcommand{\shto}{\raisebox{.3ex}{$\scriptscriptstyle\rightarrow$}\!}

\begin{document}

\title
{Commutators, paraproducts and BMO in non-homogeneous martingale settings}

\author{Sergei Treil}

\thanks{This material is based on the work supported by the National Science Foundation under the grant  DMS-0800876. Any opinions, findings and conclusions or recommendations expressed in this material are those of the author and do not necessarily reflect the views of the National Science Foundation. 
}

\address{Department of Mathematics, Brown University, 151 Thayer
Str./Box 1917,      
 Providence, RI  02912, USA }
\email{treil@math.brown.edu}
\urladdr{http://www.math.brown.edu/\~{}treil}

\date{}

\begin{abstract}

In this paper we investigate the relations between (martingale) $\BMO$ spaces, paraproducts and commutators in non-homogeneous martingale settings. Some new, and one might add unexpected, results are obtained. Some alternative proof of known results are also presented. 
\end{abstract}

\maketitle

\setcounter{tocdepth}{1}
\tableofcontents

\section*{Notation}

\begin{entry}

\item[$\X$] real line $\R$ or its subinterval. 
\item[$\cL$] Lattice of intervals in $\X$.

\item[$\1\ci I$] characteristic function of the set $I$.
\item[$\E\ci I$] averaging operator, $\E\ci I f = \1\ci I |I|^{-1}\int_I f(x) dx$.

\item[$\Delta\ci I$] martingale difference operator, $\Delta\ci I = \Bigl(\sum_{J\in\chld(I)} \E\ci J \Bigr) - \E\ci I$; here $\chld(I)$ denotes the collection of the ``children'' of $I$. 

\item[$\La f\Ra\ci I$] average value of the function $f$, $\La f\Ra\ci I = \fint_I f(x)dx := |I|^{-1} \int_I f(x) dx$.
%
%
%
%
%
%
%
%
\end{entry}

\section{Introduction and main objects}

This paper was started in attempt to understand the relations between (martingale) commutators, paraproducts and space $\BMO$. Initial hope was to cover both one-parameter and multi-parameter cases, but it became clear pretty soon that in the general, non-homogeneous case, even one-parameter situation is far from well understood.

While the results about $H^1$--$\BMO$ duality for general martingales  are well known and can be considered classical, paraproducts and commutators were studied mostly for regular $r$-adic martingales. 

In this paper several new, and one might add unexpected, results are obtained for the non-homogeneous situation. Let me list some of them here; for the definitions and exact statements the reader should look in the sections that follow.

\begin{itemize}
	\item Despite what one might expect, the condition $b\in \BMO$ is not necessary (although it is of course, sufficient) for the boundedness of the paraproduct $\pi_b$ in $L^p$. This means, in particular, that unlike the homogeneous case it is impossible to characterize $b\in\BMO$ via boundedness of commutators of the multiplication operator $M_b$  and martingale multipliers.  
	
	The condition $b\in \BMO$ however is necessary and sufficient for the boundedness of the so-called extended paraproduct $\pi_b^{(*)}$. 
	\item The necessary an sufficient condition for the $L^p$ boundedness of the paraproduct is, as one might expect, that it is enough to check the boundedness on the characteristic functions of intervals. This statement is well-known and now almost trivial for $p=2$; the result for $p\ne 2$ is new and its proof is rather complicated. 
	
	Note, that this condition depends on $p$, unlike the condition $b\in \BMO$, which guarantees the  the boundedness of $\pi^{(*)}_b$ in all $L^p$, $p\in(1, \infty)$. 
	
	\item The condition $b\in \BMO$ is, as one might expect, sufficient for the $L^p$ boundedness ($p\in(1, \infty)$) of the commutator $[M_b, T] = M_b T -TM_b$ of the multiplication operator $M_b$ and a bounded martingale transform $T$. This condition (up to some technical details) is also necessary for the boundedness of the commutator, provided that the martingale transform $T$ satisfies some ``mixing property''. 
	
This result generalizes the classical result of 
S.~Janson \cite{Janson-CommMartBMO_1981}, which gives the description of $\BMO$ via commutators in the case of regular $r$-adic martingales. The ``mixing properties'' that the martingale transform should satisfy generalize (and in the case of regular $r$-adic lattice coincide with) the notion of the non-degenerate martingale transform, considered in \cite{Janson-CommMartBMO_1981}. 

The ``mixing condition'' introduced in this paper is necessarily more complicated than the non-degeneracy condition in \cite{Janson-CommMartBMO_1981}. This is mainly due to the fact that it includes a condition that was ``hidden'' (trivially satisfied) in the homogeneous case. An example, demonstrating that this ``hidden'' condition is essential is presented in the paper. 

\item It is shown in this paper that in general non-homogeneous case the martingale difference spaces $D\ci I = \Delta\ci I L^p$ do not form the so-called \emph{strong unconditional basis} in $L^p$, $p\ne 2$ (more precisely, in the martingale Hardy space $H^p$, which is, in general for $p\in (1,\infty)$ a subspace of $L^p$ with an equivalent metric). Essentially that means that it is impossible to define an equivalent norm in $H^p$ using only the norms of martingale differences $\|\Delta\ci I f\|_p$. 

An equivalent statement is that, unlike the case $p=2$, for $p\ne 2$ there exists an unbounded in $H^p$ martingale transform $T$ (see the definition in the subsections that follow) with uniformly bounded blocks $T\ci I$. 
\end{itemize}

Few word about general setup used in the paper. We do not work here in the settings of martingale spaces, because we want to include the situation with infinite measure, like the standard dyadic lattice in $\R^n$. While getting  results in the case of infinite measure from the corresponding result in the martingale case (the finite measure) is usually pretty easy, there are some delicate situation, when one has to be careful stating the result. 
(Of course, usually after the results are stated, they are quites easy to prove, but stating 
 the results require some attention).

For example, while this is well known to specialists, it might be a surprise to a reader just casually acquainted  with martingale Hardy spaces and BMO, that for  the standard dyadic lattice in $\R$ (and in $\R^n$) one can find a function $b$ in dyadic $\BMO$ such that the martingale difference decomposition $\sum_{I\in\cD} \Delta\ci I b$ diverges a.e. I haven't seen this mentioned anywhere in the literature, probably everybody had to notice this fact for him/her-self. 

So, in this paper we work on the real line $\R$, and our $\sigma$-algebras are generated by  disjoint intervals. While practically everything can be stated and proved in the setting of arbitrary measure space, we want to avoid non-essential technical details and concentrate on main ideas. For example, at some point we will be using Fefferman--Stein maximal theorem, which is  stated and proved for  $\R^n$ but not for an arbitrary measure space. 

The settings on the real line covers the example we are mostly interested in: the case of $\R^n$ with the standard dyadic lattice and with an arbitrary Radon measure $\mu$, where the averages are taken with respect to $\mu$. Such situation is typical in the non-homogeneous harmonic analysis, cf \cite{NTV-Acta-2003,Tolsa-Acta-2003, Tolsa-H1-2001}


\subsection{Lattices, expectations and martingale differences}
Let $\mathbb X$ be either real line $\R$ or its subinterval (finite or infinite)
A lattice $\cL$ is a collection of non-trivial finite (bounded) intervals of $\X$ (say for definiteness of form $[a, b)$) with the following properties.
\begin{enumerate}
    \item $\cL$ is a union of \emph{generations} $\cL_k$, $k\in \Z$, where each generation is a collection of disjoint intervals, covering $\X$.
    \item for each $k\in\Z$, the covering $\cL_{k+1}$ is a finite refinement of the covering $\cL_k$, i.e.~each interval $I\in \cL_{k}$ is a finite union of disjoint  intervals $J\in\cL_{k+1}$. We allow the situation where  there is only one  such interval $J$ (i.e.~$J=I$); this means that $I\in\cL_k$ also belongs to the generation $\cL_{k+1}$. 
\end{enumerate}

\begin{ex*}
The main example we have in mind is the following one. Consider the space $\R^d$ with a Radon measure $\mu$ and the standard dyadic lattice. Let us represent  cubes $Q_k=Q_k^1 =[0, 2^k)^d$ by the intervals $I_k^1=[0, \mu(Q_k))\subset \R$. For each cube $Q_{k}^1$ we pick some ordering of its children (dyadic subcubes of $Q_{k}^1$ with side $2^{k-1}$) with $Q^1_{k-1}= [0,2^{k-1})^d$ being the first, and split $I_{k}^1$ into disjoint union of intervals $I_{k-1}^j$, of form $[a,b)$, $|I_{k-1}^j| =\mu(Q_{k-1}^j)$, with the ordering of the intervals
 $I_{k-1}^j$ given by the ordering of $Q_{k-1}^j$.
 
We then can order children of $Q_{k}^j$, $j\ne 1$ and represent them as subintervals of $I_k^j$, then their children, and so on.

So we have represented the standard dyadic lattice in the first ``octant'' $[0, \infty)^d$ of $\R^d$ with the measure $\mu$ by our lattice (with $\X=[0, \mu([0, \infty)^d)$), so the measure of each dyadic cube equals the length of the corresponding interval. Note, then the dyadic cubes $Q$, $\mu(Q)=0$ are ignored, the corresponding intervals are empty sets.

If the measure $\mu$ is finite, we can represent the dyadic lattice in all $\R^d$ as our lattice: in general, we can only put 2 ``octants'' on the line, but the dyadic lattice on the whole space can be represented as a finite disjoint union of our lattices.

\end{ex*}

\subsubsection{More definitions}
For an interval $I\in\cL$, let $\rk(I)$ be the rank of the interval $I$, i.e.~the largest number $k$ such that  $\cL_k\ni I$. 


For an interval $I\in \cL$, $\rk (I) =k$ a \emph{child} of $I$ is an interval $J\in\cL_{k+1}$ such that $J\subset I$ (note that by the definition of $\rk(I)$ we cannot have $J=I$, so we can write $J\subsetneqq I$). The collection of all children of $I$ is denoted by $\chl(I)$.

Let $r\in \Z$. We will call a lattice $\cL$ a $r$-adic lattice, if every generation $\cL_k$ consists of intervals of equal length, and the generation $\cL_{k+1}$ is obtained from $\cL_k$ by dividing every interval $I\in\cL_k$ into $r$ equal subintervals.

When $r=2$ we have a dyadic lattice in $\R$; if $r=2^d$, the lattice represents a dyadic lattice in $\R^d$.

We say, that a lattice $\cL$ is \emph{homogeneous}
if
\begin{enumerate}
\item Each interval $I\in \cL_{k}$ is a union of \emph{at most} $r$ ($r<\infty$) intervals $J\in \cL_{k+1}$
\item There exists a constant $K<\infty$ such that $|I|/|J|\le K$ for every $I\in \cL_k$ and every $J\in \cL_{k+1}$, $J\subset I$.  
\end{enumerate}

We say that the lattice $\cL$ is \emph{proper}, if any interval $I\in \cL_{k+1}$ is a \emph{proper} subinterval of an interval $J\in\cL$. in this case any interval $I\in\cL$ belongs to a unique generation.

\subsubsection{Conditional expectations and martingale differences} For an interval $I\in \cL$ let $\E\ci I$ be the averaging operator,
\[
\E\ci I f:= \left( |I|^{-1} \int_I f(x)dx \right) \1\ci I =: \La f\Ra\ci I \1\ci I
\]
and let $\E_k$ be the ``conditional expectation'',
\[
\E_k f = \sum_{I\in\cL_k} \E\ci I f.
\]
Consider martingale differences $\Delta\ci I$, $\Delta_k$
\[
\Delta\ci I = \Bigl( \sum_{J\in\chl(I)} \E\ci J  \Bigr)-\E\ci I, \qquad \Delta_k = \E_k - \E_{k-1} 
= \sum_{I\in \cL: \rk(I) = k-1} \Delta\ci I 
\]
(note that we cannot write $\Delta_k = \sum_{I\in \cL_{k-1}} \Delta\ci I$ here). 


Let $\mathfrak A_k$ be the $\s$-algebra generated by $\cL_k$ (i.e.~countable unions of intervals in $\cL_k$).  Let $\mathfrak A_\infty$
be the smallest $\s$-algebra containing all $\mathfrak A_k$, $k\in \Z$, and let $\mathfrak A_{-\infty}$ be the largest $\s$-algebra containing in all $\mathfrak A_k$, $\mathfrak A_{-\infty} = \cap_{k\in\Z} \mathfrak A_k$.

The structure of $\s$-algebras $\mathfrak A_{\infty}$ and $\mathfrak A_{-\infty}$ is easy to understand. Thus, $\mathfrak A_{-\infty}$ is the $\s$-algebra generated by all the intervals $I$ of form
\[
I =\bigcup_{k\in Z} I_k, \qquad \text{where}\quad I_k\in\cL_k, \ I_{k}\subset I_{k-1}.
\]
Note that $\X$ is a disjoint union of such intervals $I$ and at most countably many points (we might need to add left endpoints to the intervals $I$, if they happen to be open intervals). It is possible that there is only one such $I$,  $I=\X$, so the $\s$-algebra $\mathfrak A_{-\infty}$ is trivial. Let us denote the collection of such intervals $I$ by $\mathfrak A_{-\infty}^0$.
Define
\begin{equation}
\label{Aif}
\Aif := \{ I\in \Aio: |I|<\infty \};  
\end{equation}
``fin'' here is to remind that the set consists of intervals of finite measure.

For example, in the case of the standard dyadic lattice in $\R$, we have  that $\mathfrak A_{-\infty}^0 =\{[0, \infty), (-\infty, 0)\}$ and so $\Aif =\varnothing$.  

Instead of describing $\mathfrak A_\infty$, let us describe the corresponding measurable functions. Namely, a function $f$ is $\mathfrak A_\infty$-measurable, if it is Borel measurable and it is constant on intervals $I$
\[
I =\bigcap_{k\in\Z} I_k, \qquad \text{where}\quad I_k\in\cL_k, \ I_{k}\subset I_{k-1}.
\]
Clearly, such intervals $I$ do not intersect, so there can only be countably many of them.
Note, that if we assume that for every $x\in \X$
\begin{align}
\label{1.1}
\lim_{k\to +\infty}  |I_k(x)|  = 0,
\end{align}
where $I_k(x)$ is the unique interval in $\cL_k$ containing $x$,
then $\mathfrak A_\infty$ is the Borel $\s$-algebra.

\subsection{Martingale difference decomposition of \texorpdfstring{$L^p$}{L**p} spaces} In this paper we always assume that all functions are $\mathfrak A_\infty$-measurable.

One can easily see that
\[
\sum_{\substack{I\in\cL\\ m\le \rk(I) <n} }\Delta\ci I = \sum_{m<k\le n} \Delta_k = \E_n - \E_m .
\]
Note that for any $f\in L^p$ (we assumed here that all the functions are $\mathfrak A_{-\infty}$-measurable)
\[
\E_n f\to f \qquad\text{as} \ n\to +\infty
\]
where the convergence is a.e.~(for $p\in[1, \infty]$), and in the $L^p$ norm for $p\in [1, \infty)$.

To compute the limit $\E_m f$ as $m\to -\infty$, we notice that for a bounded compactly supported $f$ we can estimate $|E\ci I f | \le C /|I|$, so if $|I_n|\to \infty $ as $n\to \infty$, then for such functions and for $p\in (1, \infty]$
\[
\lim_{n\to\infty} \|\E\ci{I_n} f \|_p  = 0.
\]
Since bounded compactly supported functions are dense in $L^p$, $p\in [1, \infty)$, and operators $\E_n$ are contractions in $L^p$, we get applying $\e/3$ Theorem, 
that for $f\in L^p$, $p<\infty$
\[
\E_{-n} f \to \sum_{I \in \Aif } \E\ci I f  \qquad \text{as} \quad n\to \infty  ,
\]
where the convergence is in $L^p$ for $p\in (1, \infty)$ and in a weaker sense (say $L^1$ convergence on compacts) for $p=1$.

Therefore any function $f\in L^p$, $p\in (1, \infty)$ can be represented as $L^p$ convergent series
\begin{align}
\label{MartDiffDecomp}
f = \sum_{I\in \cL} \Delta\ci I f + \sum_{I \in \Aif} \E\ci I f
& =\sum_{k\in \Z} \Delta_k f + \sum_{I \in \Aif } \E\ci I f
\\  \notag
& = \sum_{k\in \Z} \Delta_k f  +\Delta_{-\infty} f;
\end{align}
we use the notation $\Delta_{-\infty} = \E_{-\infty} := \sum_{I \in \Aif } \E\ci I$ here.

We had shown the convergence of partial sums $\sum_{m}^n$, but in fact the convergence of the series is unconditional (independent of ordering).

\subsection{Martingale Hardy spaces}
\label{MartHardySp}
Everything in this subsection is well known, we present it only for the convenience of the reader. 

Let us recall  the  classical result by D.~Burkhold\-er, which in our notation can be stated as follows.

\begin{thm}[D.~Burkholder]
\label{t-Burkh}
Let $f$, $g$ be two locally integrable functions on $\X$ such that  a.e.~on $\X$
\[
|\Delta\ci I f | \le |\Delta\ci I g|\quad \forall I \in \cL , \qquad \text{and} \qquad |\E\ci I f | \le |\E\ci I g| \quad \forall I \in \mathfrak A_{-\infty}^0, |I|<\infty .
\]
Then
\[
\|f\|_p \le (p^* - 1) \|g\|_p,
\]
where $p^*= \max\{p, p'\}$, $1/p+ 1/p' =1$.
\end{thm}

In \cite{Burkholder_LNM-1991} this theorem was proved for arbitrary discrete time martingales, which immediately gives the above theorem in the special case $|\X| = 1$, $\mathfrak A_k =\{ \X \}$ for $k\le 0$. The general statement can be easily obtained from this special case by easy and standard reasoning, which we skip.

Burkholder's theorem implies that for $|\alpha_k|=1$
\[
\frac1C \|f|_p \le
\biggl\| \sum_{k\in \Z\cup\{-\infty\} }\alpha_k \Delta_k f  \biggr\|_p
\le C \|f\|_p,
\]
where $C =p^*-1$.

Taking for $\alpha_k$ independent Bernoulli random variables, taking values $\pm1$ with probability $1/2$, and taking expectation one gets
\[
\frac1{C^p} \|f|_p^p \le
\int_\Omega \int_\X
\biggl|  
\sum_{k\in \Z\cup\{-\infty\} }\alpha_k(\om) \Delta_k f(x)  \biggr|^p dx dP(\om)\le C^p \|f\|_p,
\]
Changing order of integration, and noticing that by Khinchine inequality for any sequence of $x_k\in\C$, the averages
\[
\Biggl(\int_\Omega  
\biggl|  
\sum_{k }\alpha_k(\om) x_k  \biggr|^p  dP(\om)
\Biggr)^{1/p}
\]
and
\[
\Biggl(\int_\Omega  
\biggl|  
\sum_{k }\alpha_k(\om) x_k  \biggr|^2  dP(\om)
\Biggr)^{1/2}
= \Bigl( \sum_{k} |x_k|^2 \Bigr)^{1/2}
\]
are equivalent with constants depending only on $p$, we can see that the quantity $\| \wt S f \|_p$, where $\wt S f$ is the so-called extended square function
\begin{equation}
\label{SqFunct}
\wt S f (x) = \biggl( \sum_{k\in\Z\cup\{-\infty\}} |\Delta_k f(x)|^2 \biggl)^{1/2} ,
\end{equation}
defines an equivalent norm in $L^p$ (recall that we assume that all functions are $\mathfrak A_\infty$  measurable).

In particular, this implies that for $f\in L^p$ the sum in \eqref{MartDiffDecomp} converges unconditionally (independently of ordering) in $L^p$. Note, that if for a formal sum $f$ of form \eqref{MartDiffDecomp} we have $\wt Sf\in L^p$, then the series converges unconditionally in $L^p$, so $L^p$, $p\in (1, \infty)$ is isomorphic to the set of formal series  \eqref{MartDiffDecomp} with $\wt Sf\in L^p$

Let us also introduce the classical square function $S$, where we do not add the term $|\Delta_{-\infty} f |^2$, 
\begin{equation}
\label{SqFunct-1}
 S f (x) = \biggl( \sum_{k\in\Z} |\Delta_k f(x)|^2 \biggl)^{1/2} .
\end{equation}

The situation for $p=1$ is more interesting. Recall the classical result of Burges Davis \cite{Davis-IntegrSqFunct_1970} comparing maximal function with the square function.
Let us recall that the maximal function $M=M_\cL$ is defined by
\[
M f (x) := \sup_{I\in\cL : x\in I} | \E\ci I f | = \sup_{k\in\Z} |\E_k f(x) |
\]

\begin{thm}[B.~Davis, 1970]
\label{t.Sq-Max}
Let $M=M_\cL$ be the maximal function defined above, and let $S(f)$ be the square function defined by \eqref{SqFunct}. Then
\[
\frac1C \| Mf\|_1 \le \|\wt S f\|_1 \le C \| Mf\|_1,
\]
where $C$ is an absolute constant.
\end{thm}
\begin{rem*}
The theorem in \cite{Davis-IntegrSqFunct_1970} was proved for general discrete time martingales, and in our case it can be directly applied in to the situation $|\X|=1$, $\mathfrak A_{k} = \{\X\}$ for $k\le 0$, $\E\ci\X f=0$. However, the general case can be easily obtained from here by a standard reasoning, which we skip here.
\end{rem*}

Note, that by the Lebesgue differentiation theorem $\|f\|_1\le \|Mf\|_1$. Therefore, if $\|Sf_n\|_1\to 0$, then $\|f_n\|_1\to 0$, so if $Sf\in L^1$, then the martingale difference decomposition  \eqref{MartDiffDecomp} converges unconditionally in $L^1$.

\begin{df*}
The martingale extended Hardy space $\wt H^1$ is  set of all functions $f\in L^1$  such that $\wt Sf\in L^1$ (equivalently, $M f \in L^1$), equipped with the norm $\| f\|\ci{\wt H^1} =  \|\wt Sf\|_1$. 

The Hardy space $H^1$ consists of all the functions in $\wt H^1$ such that $\E\ci I f =0$ for all $I\in\Aif$ (with the norm given by $\|Sf\|_1$). Note, that $\|Mf\|_1$ also gives an equivalent norm on $H^1$. 
\end{df*}

\begin{rem*}
For $p\in(1, \infty)$ the extended martingale Hardy space $\wt H^p$ is also defined as  the space of all locally integrable functions $f$ such that  $\wt Sf\in L^p$, with the norm $\|f\|_{\wt H^p} =\|Sf\|_p$. While, as we discussed above, for $p\in(1, \infty)$ the space $\wt H^p$ is isomorphic to  $L^p$, we will use the notation $\wt H^p$ as well (for example, to emphasize that we are using a different norm). 

Finally, the spaces $H^p$ are defined as subspaces of $\wt H^p$ consisting of functions $f$ such that $\E\ci I f =0$ for all $I\in\Aif$. 
\end{rem*}

\subsection{Martingale transforms and martingale multipliers}Let $D\ci I :=  \Delta\ci I L^2$. A \emph{martingale transform} is a linear transformation $T$
\[
T\Bigl( \sum_{I\in\cL} \Delta\ci I f \Bigr) = \sum_{I\in\cL} T\ci I (\Delta\ci I f)
\]
where $T\ci I :D\ci I \to D\ci I$. We also assume that $T\E\ci I f =0$ for all $I\in \Aif$. 

Such operators are well defined for finite sums: for now we will not assume the boundedness of $T$.

If all operators $T\ci I$ are multiples of identity, the corresponding martingale transform is called a \emph{martingale multiplier}.

\subsection{Paraproducts}
For a function $b$ let us consider the multiplication operator $M_b$, $M_b f =bf$. We do not assume here that $M_b$ is bounded in $L^2$ (i.e.~that $b\in L^\infty$. For our purposes, it is enough to assume that $b\in L^1_{\text{loc}}$, so $\La M_b f, g\Ra$ is well defined for $f$ and $g$ with finite martingale decompositions, i.e.~for finite sums  
\begin{equation}
\label{1.3}
    f = \sum_{I\in\cL} \Delta\ci I f + \sum_{I\in \Aif} \E\ci I f, \qquad g= \sum_{I\in\cL} \Delta\ci I g +\sum_{I\in \Aif} \E\ci I g.
\end{equation}

\subsubsection{The ``infinite measure'' case} Let us first consider the situation when $\Aif=\varnothing$. 

In this case, as it was discussed above, the space $L^2$ is decomposed in the orthogonal sum of subspaces $D\ci I$, $I\in \cL$.

Consider the decomposition of the operator $M_b$ in this orthogonal basis
\[
M_b f =\sum_{I\in \cL}\sum_{J\in \cL} \Delta\ci I M_b \Delta\ci J f.
\]
This sum can be split into 3 parts: over $I\subsetneqq J$, $J\subsetneqq I$ and $I=J$ respectively.

The first sum is called the \emph{paraproduct} and is denoted as $\pi_b f$; the corresponding operator $\pi_b$ is also called the paraproduct. Since for $I\subsetneqq J$
\[
\Delta\ci I (b\Delta\ci J f) = (\Delta\ci I b)( \Delta\ci J f),
\]
we can write
\begin{equation}
\label{1.4}
\pi_b f = \sum_{I, J\in\cL: I\subsetneqq J}  \Delta\ci I (b\Delta\ci J f) =  \sum_{I, J\in\cL: I\subsetneqq J} (\Delta\ci I b)( \Delta\ci J f) = \sum_{I\in \cL} (\Delta\ci I b)( \E\ci I f);
\end{equation}
the last equality follows from the fact that for fixed $I\in \cL$,
\[
\sum_{J\in\cL: I\subsetneqq J} ( \Delta\ci J f)\,\1\ci I =\E\ci I f.
\]

The second sum (over $J\subsetneqq I$) is  $\pi^*_b f$, where $\pi^*_b$ is the dual of $\pi_b$ with respect to the standard linear duality $\La f, g\Ra = \int fg$. This can be easily seen from the fact that $\La \E\ci I f, g\Ra = \La f, \E\ci I g\Ra$ and so $\La \Delta\ci I f, g\Ra = \La f, \Delta\ci I g\Ra$.

The third sum (over $I=J$) is the ``diagonal'' term denoted as $\Lambda_b f$. It is easy to see that
\[
\Lambda_b f =\sum_{I\in\cL} \Delta\ci I \left( b \Delta\ci I f\right)
\]
This diagonal term commutes with all martingale multipliers, so it can be ignored when one studies commutators of $M_b$ with martingale multipliers.  

In the situation when all intervals $I\in\cL$ have at most $2$ children, any martingale transform is a multiplier, so in this case it is enough to consider decomposition of $M_b$ as
\begin{equation}
\label{1.4a}
M_b = \pi_b + \pi_b^* + \Lambda_b;
\end{equation}
where we can ignore the term $\Lambda_b$ when studying commutators with martingale transforms.

In a general situation, we can only ignore a term that is a martingale multiplier, so a different decomposition is needed. To present this decomposition we need the following lemma, which gives us a formula for $\pi_b^*$.

\begin{lm}
\label{l.pi^*}
The (formal) dual $\pi_b^*$ of $\pi_b$ with respect to the standard linear duality is given by
\begin{align*}
\pi^*_b f & = \sum_{I\in\cL} \E\ci I \left( (\Delta\ci I b)(\Delta\ci I f) \right) \\
& =
 \sum_{I\in\cL}\E\ci I \left( (b - \E\ci I b)(\Delta\ci I f) \right)
 =
 \sum_{I\in\cL} \E\ci I \left( b (\Delta\ci I f) \right)
.
\end{align*}
The word ``formal'' here means that the equality $\La \pi_b f , g \Ra = \La  f ,\pi_b^* g \Ra$ holds for all finite sums $f=\sum_{I\in \cL} \Delta\ci I f$, $g=\sum_{I\in \cL} \Delta\ci I g$.
\end{lm}

\begin{proof}
It is easy to see that $\La \E\ci I f, g \Ra = \La  f,\E\ci I g \Ra$, and therefore so $\La \Delta\ci I f, g \Ra = \La  f,\Delta\ci I g \Ra$. Using these identities and the fact that $(\Delta\ci I b)(\E\ci I f) = \Delta\ci I (b\E\ci I f)$, we get
\begin{align*}
\La \pi_b f , g \Ra & = \sum_{I\in\cL} \left\La \Delta\ci I (b \,\E\ci I f), g\right\Ra \\
& =
\sum_{I\in\cL} \left\La b\, \E\ci I f,  \Delta\ci I g\right\Ra \\
& =
\sum_{I\in\cL} \left\La  \E\ci I f, b\, \Delta\ci I g\right\Ra \\
& =
\sum_{I\in\cL} \left\La  f,   \E\ci I(b\,  \Delta\ci I g)\right\Ra
\end{align*}
To complete the proof it remains to show that
\[
\E\ci I(b\,  \Delta\ci I g) =
\E\ci I \left( (b - \E\ci I b)(\Delta\ci I g) \right)
=
\E\ci I \left( (\Delta\ci I b)(\Delta\ci I g) \right) ,  
\]
which we leave as an exercise for the reader.
\end{proof}

To give an alternative (to \eqref{1.4a}) decomposition of $M_b$ let us notice that
\begin{align*}
\Delta\ci I \left( b \Delta\ci I f\right) & = \Delta\ci I \left( ( b - \E\ci I b) \Delta\ci I f\right) + (\E\ci I b)\Delta\ci I f
\\
& = \Delta\ci I \left( (\Delta\ci I b) \Delta\ci I f\right)  + (\E\ci I b)\Delta\ci I f .
\end{align*}
Therefore, we can decompose $\Lambda_b = \Lambda_b^1 + \Lambda_b^0$, where
\begin{align}
\label{Lam_b^1}
\Lambda_b^1 f & = \sum_{I\in\cL}  \Delta\ci I \bigl[ (\Delta\ci I b)(\Delta\ci I f) \bigr], \\
\label{Lam_b^0}
\Lambda_b^0 f & = \sum_{I\in\cL} (\E\ci I b) (\Delta\ci I f).
\end{align}
Note, that $\Lambda_b^0$ is a martingale multiplier, so it commutes with all martingale transforms. 

Defining 
\begin{equation}
\label{pi(*)}
\pi^{(*)}_b := \pi_b^* + \Lambda_b^1,
\end{equation}
we can  decompose the multiplication operator $M_b$, $M_b f := bf$  as
\begin{equation}
\label{1.5}
M_b = \pi_b + \pi^{(*)}_b + \Lambda_b^0,
\end{equation}

\begin{lm}
\label{l.pi(*)}
\begin{equation}
\label{pi(*)-1}
\pi_b^{(*)} f = \sum_{I\in\cL} (\Delta\ci I b)(\Delta\ci I f). 
\end{equation}
\end{lm}
\begin{proof}
Notice that 
\[
(\Delta\ci I b)(\Delta\ci I f) = \E\ci I \left(  (\Delta\ci I b)(\Delta\ci I f) \right) +
\Delta\ci I \left(  (\Delta\ci I b)(\Delta\ci I f) \right)  .
\]
Taking the sum over all $I\in\cL$ we get in the right side $\pi^*_b f + \Lambda_b^1 f$, which proves the lemma. 
\end{proof}

There is an alternative, probably a more natural way to get the decomposition \eqref{1.5}. Namely, let us consider  
the product $bf$, which can be written as
\[
\sum_{I, J\in \cL} (\Delta\ci I b)( \Delta\ci J f)
\]
(let us not worry about convergence here and assume that the sums in the martingale difference decompositions of $f$ and $b$ are finite).

Let us split the above sum into 3 parts, over the sets $I\subsetneqq J$, $J\subsetneqq I$ and $I=J$ respectively.

The first sum gives us the \emph{paraproduct}  $\pi_b f$
\begin{equation}
\label{para-alt}
 \sum_{I, J\in\cL: I\subsetneqq J} (\Delta\ci I b)( \Delta\ci J f) = \sum_{I\in \cL} (\Delta\ci I b)( \E\ci I f) =: \pi_b f ,
\end{equation}
cf.~\eqref{1.4}.

The second sum (over $J\subsetneqq I$) can be written as $\pi_f b$, so using \eqref{1.4} with $f$ and $b$ interchanged and recalling the definition of $\Lambda_b^0$, see \eqref{Lam_b^0}, we get
\[
 \sum_{I,J\in\cL:J\subsetneqq I} (\Delta\ci I b)( \Delta\ci J f) = \sum_{J\in\cL} (b\ci J)( \Delta\ci J f) =: {\Lambda}_b^0 f.
\]

Finally,  the last sum gives us
\[
 \sum_{I\in\cL} (\Delta\ci I b)(\Delta\ci I f) =: \pi^{(*)}_b f,
\]
see \eqref{pi(*)-1}.


\begin{rem*}
Note, that if $\cL$ is the standard dyadic lattice, then $\E\ci I \left( (\Delta\ci I b)(\Delta\ci I f) \right)  =
(\Delta\ci I b)(\Delta\ci I f))$, so $\pi^{(*)}_b = \pi^*_b$. This fact was used, for example, in \cite{Blasco-Pott_2005}.  
\end{rem*}

\subsubsection{Paraproducts in general case}
Let us now consider the general case, when $\Aif\ne\varnothing$. Consider the decompositions
\begin{equation*}
    f = \sum_{I\in\cL} \Delta\ci I f + \sum_{I\in \Aif} \E\ci I f, \qquad g= \sum_{I\in\cL} \Delta\ci I g +\sum_{I\in \Aif} \E\ci I g, 
\end{equation*}
and let us  decompose $\La bf, g\Ra$. Note that for a fixed $I\in\cL$
\[
\Bigl\La b \Biggl( \sum_{J\in\cL: J\supsetneqq I} \Delta\ci J f + \sum_{J\in \Aif:J\supset I} \E\ci I f \Biggr) , \Delta\ci I g \Bigr\Ra = 
\La (\Delta\ci I b) \E\ci I f, \Delta\ci I g\Ra = \La \pi_b f, \Delta\ci I g\Ra, 
\]
where, as above
\begin{equation}
\label{para1}
\pi_b f := \sum_{I\in\cL} (\E\ci I f)(\Delta\ci I b). 
\end{equation}
Similarly, 
\[
\Bigl\La b \Delta\ci I f, \Biggl( \sum_{J\in\cL: J\supsetneqq I} \Delta\ci g f + \sum_{J\in \Aif:J\supset I} \E\ci I g \Biggr)  \Bigr\Ra = 
\La \Delta\ci I f , (\Delta\ci I b) \E\ci I g \Ra = \La \Delta\ci I f, \pi_b   g\Ra . 
\]
As we discussed above
\[
\sum_{I\in\cL} \La b \Delta \ci I f, \Delta\ci I g \Ra = \La \Lambda_b f, g \Ra, 
\]
where
\begin{equation}
\label{La_b}
\Lambda_b f := \Delta\ci I (b \Delta\ci I f). 
\end{equation}
The only terms in $\La f, g\Ra$ that we did not count yet, are the terms with $I, J\in\Aif$, which give us the remainder
\[
\sum_{I\in \Aif} \La b \E\ci I f , \E\ci I g \Ra = 
\Bigl\La \sum_{I\in \Aif} (\E\ci I b) \E\ci I f , \E\ci I g \Bigr\Ra =: \La R_b f, g\Ra. 
\]

So, the multiplication operator $M_b$ can be decomposed as 
\[
M_b = \pi_b^* + \Lambda_b + \pi_b + R_b, 
\]
where the paraproduct $\pi_b$ is defined by \eqref{para1}, $\pi_b^*$ is its adjoint, $\Lambda_b$ is defined by \eqref{La_b}, and 
\begin{equation}
\label{R_b}
R_b f = (\E_{-\infty} b) (\E_{-\infty} f) = \sum_{I\in\Aif} (\E\ci I b)(E\ci I f). 
\end{equation}

Note, that Lemma \ref{l.pi^*} remains true in the general case as well: the  proof is exactly  the same. Also, nothing changes in the decomposition $\Lambda_b = \Lambda_b^0 + \Lambda_b^1$, because we can investigate this decomposition separately in each block $D\ci I$, and these blocks know nothing about $\Aif$. Finally, the proof of Lemma \ref{l.pi(*)} works in the general case without any changes. 

Summarizing we can state the following proposition.

\begin{prop}
\label{p1.2}
The multiplication operator $M_b$ is represented (at least formally) as
\[
M_b  -R_b  = \pi_b^{(*)} + \Lambda_b^0 + \pi_b = \pi_b^{*} + \Lambda_b + \pi_b = \pi_b^{*} + \Lambda_b^0 + \Lambda_b^1 + \pi_b
\]
where
\begin{align*}
\Lambda_b f & = \sum_{I\in\cL} \Delta\ci I (b \Delta\ci I f ) , \qquad \Lambda_b = \Lambda_b^0 + \Lambda_b^1, \\
\Lambda_b^0 f & = \sum_{I\in\cL} (\E\ci I b) (\Delta\ci I f), \\
\Lambda_b^1 f & = \sum_{I\in\cL}  \Delta\ci I \bigl[ (\Delta\ci I b)(\Delta\ci I f) \bigr]  .
\end{align*}

\end{prop}

\section{
Triebel--Lizorkin type spaces.}

This part is devoted to the investigation of the ``coefficient space'' of the spaces $H^p$. We are mostly interested in the spaces with $q=2$, but since since the result for $q\ne 2$ are often obtained with little or no extra effort, we consider the case of general $q$ here.

The notation  $\gb_p^q$ is chosen by the analogy with the notation $\fb_p^{\alpha, q}$ for Triebel--Lizorkin spaces, see for example \cite{Frazier-Jawerth-1990}. We use a different scaling here, so to avoid the confusion we use the different notation.  Also, we do not use smoothness parameter $\alpha$ (we do not need it in what follows, and frankly, it is not completely clear what should be the correct smoothness in the general non-homogeneous case). For the standard dyadic lattices in $\R^d$ our spaces $\gb_p^q$ are isomorphic to $\fb_p^{0, q}$, with isomorphism given by rescaling of the entries.

\subsection{Triebel--Lizorkin type spaces \texorpdfstring{$\gb_p^q(\cL)$}{}}

Let $\cL$ be a lattice. For $1\le p, q< \infty$ define the sequence spaces $\gb_p^q(\cL)$ consisting of sequences $s=\{s\ci I\}_{I\in\cL}$ such that
\[
\| s\|_{\gb_p^q(\cL)} := \biggl\| \biggl( \sum_{I\in\cL} |s\ci I |^q  \1\ci I \biggr)^{1/q} \biggr\|_{L^p}.
\]
For $p=\infty$ the norm is defined using BMO-like norm
\[
\|s\|_{\gb^{q}_\infty(\cL)} := \sup_{J\in\cL} \left( \frac{1}{|J|} \int_J \sum_{I\in \cL, \, I\subset J}
 |s\ci Q|^q \1_Q  \right)^{1/q}
\]
Formally, one can define the whole scale of spaces $\gb^{q,(r)}_{\infty}(\cL)$, $1\le r<\infty$,
\[
\|s\|_{\gb^{q,(r)}_{\infty}(\cL)} := \sup_{J\in\cL} \left( \frac{1}{|J|} \int_J \Bigl(\sum_{I\in \cL, \, I\subset J}
 |s\ci Q|^q \1_Q \Bigr)^{r/q} \right)^{1/r},
\]
but it will be shown later that the norms are equivalent for $1\le r <\infty$.

To shorten the notation, we will omit $\cL$ and use the notation $\gb_p^q$ instead of $\gb_p^q(\cL)$, when  it is clear from
the context what the lattice $\cL$ is.

The spaces $\gb_p^q(\cL)$ can be naturally identified with the subspaces of $L^p(\ell^q)$ ($L^p$ with values in $\ell^q$). Namely, for a sequence $s= \{s\ci I\}_{I\in\cL}$ define functions
\[
f_k = \sum_{I\in \cL: \rk(I) = k} s\ci I \1\ci I,  \qquad k\in \Z,
\]
and let
\[
f(x, k) = f_k(x), \qquad k\in\Z, \quad x\in \R.
\]
Then clearly, for $1\le p, q<\infty$
\[
\| s\|_{\gb_p^q(\cL)} = \|f\|_{L^p(\ell^q)} := \left(\int \| f(x, \fdot)\|_{\ell^q}^p dx\right)^{1/p}.
\]

Thus, the space $\gb_p^q(\cL)$, $1\le p, q<\infty$ can be naturally identified with the subspace of $L^p(\ell^q)$ consisting of functions $f$ such that $f(\fdot, k)$ is constant on intervals $I\in\cL$, $\rk(I) =k$, and such that $f(x, k) = 0$ if  there is no interval $I\in\cL$, $\rk(I)=k$ containing $x$ (recall that $\rk(I)$ is the largest integer $k$ such that $I\in \cL_k$, so the condition $I\in \cL_k$ does not mean that $\rk(I)=k$).

We will routinely switch between the function and sequence representation of elements of $\gb_P^q$, so $f\in \gb_p^q$ as a sequence $\{f\ci I \}\ci{I\in\cL}$ or as the corresponding function $f(\fdot, \fdot) \in L^p(\ell^q)$.

We will also need the notion of the \emph{coordinate projection}  of $f\in \gb_p^q$. Namely, for $\cE\subset \cL$ define the coordinate projection $f\ci \cE$ by
\begin{equation}
\label{coord.proj}
f\ci \cE = \{ f\ci I \}\ci{I \in \cE}   
\end{equation}
(meaning that entries corresponding to $I\notin \cE$ are $0$). In the function representation this can be written as
\begin{equation}
\label{coord.proj.1}
f\ci \cE(\fdot, k) = f(\fdot, k)\cdot \Bigl( \sum_{I\in\cL:\rk(I) =k} \1\ci I\Bigr), \qquad k\in\Z.
\end{equation}

For $f\in L^p(\ell^q)$ define the vector Hardy--Littlewood maximal function $f^*$
\[
f^*(x, k) = \sup_{I\ni x} \, \frac1{|I|}\int_I |f(s, k)|ds.
\]

We will need the following well-known theorem
\begin{thm}[Fefferman--Stein, \cite{Fef-Stein_Max_Inequalities_1971}]
\label{tFS}
Let $f\in L^p(\ell^q)$, $1<p, q<\infty$. Then
\[
\|f^*\|_{L^p(\ell^q)} \le C\|f \|_{L^p(\ell^q)},
\]
where $C$ depends only on $p$ and $q$.
\end{thm}

The following fact is well known.

\begin{prop}
For $1<p,q<\infty$ the dual space $(\gb_p^q(\cL))^*$ is isomorphic to $\gb_{p'}^{q'}(\cL)$, where $1/p+1/p'=1$, $1/q+1/q'=1$, and the pairing is the standard one
\begin{equation}
\label{n2.1}
\La f, g\Ra = \int_\X \sum_k f(x,k) g(x, k) dx = \int_\X \sum_{I\in\cL} f\ci I g\ci I \1\ci I = \sum_{I\in\cL} f\ci I g\ci I |I|.
\end{equation}
Note that only claim that the norm in $\gb_{p'}^{q'}(\cL)$ is equivalent to the norm in the dual space (except the trivial case $p=q=2$ when the norms  coincide).
\end{prop}
For the sake of completeness we present the proof of this proposition.
\begin{proof}
Since $(L^p(\ell^q))^* = L^{p'}(\ell^{q'})$ (for $1<p, q<\infty$), any $g \in \gb_{p'}^{q'}(\cL)$ define a bounded linear functional $L$ on $\gb_p^q(\cL)$, and $\|L\|\le \|g\|_{\gb_{p'}^{q'}(\cL)}$.

On the other hand, if $L$ is a bounded linear functional on $\gb_p^q(\cL)$, it can be extended by Hahn--Banach Theorem to a bounded linear functional on $L^p(\ell^q)$, which can be represented by a function $\wt g\in L^{p'}(\ell^{q'})$,
\[
L(f) = \int \sum_k f(x, k) \wt g(x, k) dx, \qquad \forall f\in \gb_p^q(\cL)
\]
Note that functional $L$ will not change if we replace the function $\wt g$ by its ``orthogonal'' projection $g$ onto $\gb_{p'}^{q'}(\cL)$,
\[
g(x, k) =
\left\{ \begin{array}{ll}
|I|^{-1} \int_I \wt g (s, k) ds, \qquad &\text{if } \rk(I)=k, \text{ and } x\in I \\
0 &\text{if } \not \exists I\ni x, \ \rk(I)=k .
\end{array}\right.
\]
Clearly $|g|\le (\wt g)^*$, so by the Fefferman--Stein maximal theorem (Theorem \ref{tFS})
\[
\|g\|_{L^{p'}(\ell^{q'})} \le C \| \wt g\|_{L^{p'}(\ell^{q'})}.
\]
\end{proof}

Dual of $\gb_1^q$ is given by the following theorem.

\begin{thm}
\label{t2.3}
Let $1<q<\infty$. Then the spaces $\gb_\infty^{q, (r)}$, $1\le r<\infty$ do not depend on $r$, and the corresponding norms are equivalent. Moreover, the dual space $(\gb_1^q)^*$ is isomorphic to $\gb_\infty^{q'}$; here again $1/q+1/q'=1$ and the pairing is given by \eqref{n2.1}.  
\end{thm}

\begin{proof}
Let $g\in \gb_\infty^{q', (1)}$. We want to show that
\[
|\La f, g\Ra | \le C \|f\|_{\gb_1^q} \|g\|_{\gb_\infty^{q', (1)}}.
\]

It is sufficient to prove this inequality on a dense set of functions $f$ for which the corresponding sequence $\{s\ci I\}\ci{I\in\cL}$ has finitely many non-zero terms.

Let $E_k:=\{ x\in\R: \| f(x, \fdot) \|_{\ell_q}> 2^k\}$, and let $\cE_k := \{ I\in\cL: I \subset E_k\}$. Note than $E_k$ is a  finite disjoint union of maximal intervals $I\in \cE_k$, maximal meaning that there is no interval in $\cE_k$ for which $I$ is a proper subinterval.  

One can easily see (see Fig.~\ref{fig1} )that
\[
\sum_{k\in\Z} 2^k |E_k|\le 2 \int \| f(x, \fdot) \|_{\ell_q} dx.
\]
Since $\cL$ is a disjoint union of the sets $\cE_k\setminus\cE_{k+1}$ we write
\[
f = \sum f\ci{\cE_k\setminus \cE_{k+1}}
\]
where the functions $f\ci{\cE_k\setminus \cE_{k+1}}$ are defined in the sequence representation $f=\{f\ci I\}\ci{I\in\cL}$  by
\[
f\ci{\cE_k\setminus \cE_{k+1}} = \{f\ci I\}\ci{I\in \cE_k\setminus \cE_{k+1}}.
\]

\setlength{\unitlength}{1mm}
\begin{figure}

\begin{center}
\includegraphics{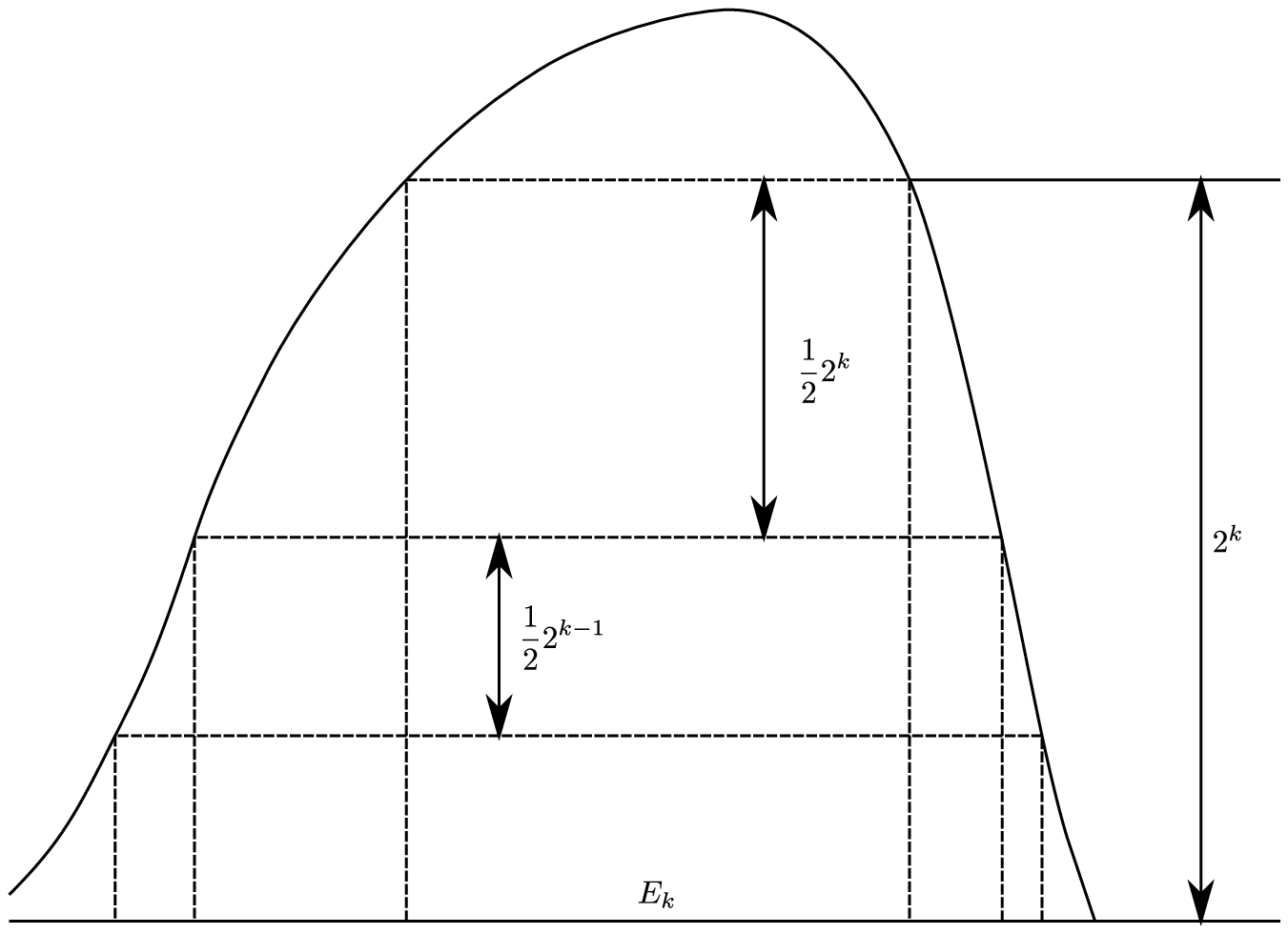}
\end{center}
\caption{Level sets}%
{\protect\label{fig1}}%
\end{figure}

Note that
\begin{equation}
\label{2.2}
\| f\ci{\cE_k\setminus \cE_{k+1}}(x, \fdot)\|_{\ell^q} \le 2^{k+1} \1\ci{\cE_k}(x) .
\end{equation}
Indeed, the estimate for $x\notin E_{k+1}$ is trivial. For $x\in E_{k+1}$ let $J$ be the maximal interval in $\cE_{k+1}$ containing $x$.  Then for this $x$
\begin{equation}
\label{2.3}
\| f\ci{\cE_{k}\setminus\cE_{k+1}}(x, \fdot) \|_{\ell^q}^q = \sum_{I\in\cE_k: I\supsetneqq J} |f\ci I|^q.
\end{equation}
Let $\wt J$ be the parent of $J$. We assume that$\wt J\in\cE_k$, because otherwise the sum is trivial and $f\ci{\cE_k \setminus \cE_{k+1}} = 0$ on $I$. Then for any $y\in\wt J$
\begin{equation}
\label{2.4}
\sum_{I\in\cE_k: I\supsetneqq J} |f\ci I|^q \le \| f\ci{\cE_{k}\setminus\cE_{k+1}}(y, \fdot) \|_{\ell^q}^q
\end{equation}
(the right side contain all the terms from the left side, plus probably some other terms).

But since $\wt J\notin \cE_{k+1}$, there exists $y\in \wt J$ such that $y\notin E_{k+1}$, and so
\[
\| f\ci{\cE_{k}\setminus\cE_{k+1}}(y, \fdot) \|_{\ell^q}^q \le 2^{k+1}.
\]
Together with \eqref{2.4} and \eqref{2.3} this inequality implies \eqref{2.2}.  


To finish the proof let us take $g\in \gb_\infty^{q',(1)}$, $\|g\|\ci{\gb_\infty^{q',(1)}} \le 1$ and estimate
\begin{align*}
|\La f, g\Ra| = \Bigl| \int\sum_{I\in\cL} f\ci I g\ci I \1\ci I dx \Bigr|
& \le \sum_{k\in\Z}
 \int \sum_{I\in\cE_k\setminus \cE_{k+1}}  |f\ci I| |g\ci I| \1\ci I dx \\
& \le \sum_{k\in\Z} \|f\ci{\cE_k\setminus\cE_{k+1}}\|\ci{L^\infty(\ell^q)}  
\|g\ci{\cE_k\setminus\cE_{k+1}}\|\ci{L^1(\ell^{q'})}  \\
& \le
\sum_{k\in\Z} 2^{k+1} |E_k| \le 4 \|f\|_{L^1(\ell^q)}.
\end{align*}
Here in the third line we used the fact that by \eqref{2.2} we have $\|f\ci{\cE_k\setminus\cE_{k+1}} \|\ci{L^\infty(\ell^q)}  \le 2^{k+1}$   and that $\| g\ci{\cE_k \setminus\cE_{k+1}} \|\ci{L^1(\ell^{q'})} \le \| g\ci{\cE_k } \|\ci{L^1(\ell^{q'})} \le |E_k|$.  

So, we have proved that a function $g\in \gb_\infty^{q',(1)}$ defines a bounded linear functional of $\gb_1^q$.

Let now $\f $ be a bounded linear functional on $\gb_1^q$. By Hahn--Banach Theorem it can be extended to a functional on $L^1(\ell^q)$, so it can be represented as
\[
\f(f) = \La f, \wt g\Ra = \int_\R \sum_{k\in\Z} f(x,k) \wt g(x, k) dx
\]
where $\wt g\in L^\infty(\ell^{q'})$, $\|\wt g\|\ci{L^\infty(\ell^{q'})} = \|\f\|$.

Let $g = \cP\ti{seq}(\wt g)$ be the projection of $\wt g$ onto the space of sequences, i.e.~let the function $g$ is given in sequence representation
by
\[
g\ci I = \fint_I \wt f(x, \rk(I)) dx .
\]
The projection $\cP\ti{seq}$, as it can be easily shown, is not bounded in  $L^\infty(\ell^{q'})$, but by Fefferman--Stein theorem (Theorem \ref{tFS}) it is bounded in $L^r(\ell^{q'})$.

Therefore, for any (finite union of intervals) $E\subset \R$ and the collection $\cE:=\{I\in\cL: I\subset E\}$
\[
\|g\ci \cE \|_{L^r(\ell^{q'})}^r \le C \|\wt g\ci \cE \|_{L^r(\ell^{q'})}^r \le C |E| \|\wt g\ci \cE \|_{L^\infty(\ell^{q'})}^r
\]
which means exactly that $g \in \gb^{q',(r)}_\infty$, $\| g \|_{\gb^{q',(r)}_\infty} \le C \|\f\|$.
\end{proof}

\subsection{Embedding theorem for \texorpdfstring{$\gb_p^q$}{g<sub>p<sup>q}}
Let $\{\alpha\ci I\}\ci{I\in\cL}$ be a collection of numbers. We are interested when the operator
\[
f\mapsto \{\alpha\ci I \La f \Ra\ci I \}\ci{I\in\cL}
\]
is a bounded operator from $L^p$ to $\gb_p^q$; recall that for a function $f$ the symbol $\La f \Ra\ci I$ denotes it average,  $\La f \Ra\ci I = \fint_I f$.

The answer to this question is well known if $p=q$, is is given by the famous Carleson Embedding Theorem, that says that a necessary and sufficient condition to the boundedness is
\[
\sup_{I\in\cL} \frac1{|I|} \sum_{J\in \cL: J\subset I} |\alpha\ci J |^q <\infty
\]
which means exactly that $\{\alpha\ci I\}\ci{I\in \cL} \in \gb_\infty^{q,(q)} = \gb_\infty^q$. This result is especially well known for $p=q=2$; the situation for $p=q$ can be obtained, as it will be shown below, by the standard comparison with maximal function.

Below, we will show that the answer is the same for all $p\in (1, \infty)$. For $p=1$ the above condition is not sufficient for the embedding, but if we replace $L^1$ by $\wt H^1$, then the result can be extended to $p=1$. 

\begin{thm}
\label{t2.4}
Let $p\in[1, \infty)$, $q\in(1, \infty)$ and let $\alpha = \{\alpha\ci I\}\ci{I\in\cL}$ be a collection of numbers. Then the operator $A_\alpha$
\[
A_\alpha f = \{\alpha\ci I \La f \Ra\ci I \}\ci{I\in\cL}
\]
is a bounded operator $\wt H^p \to \gb_p^q$ if and only if $\alpha \in \gb_\infty^q$.
\end{thm}

\begin{proof}
To prove the necessity of the condition $\alpha \in \gb_\infty^q$ we just have to test the embedding operator on the functions $\1\ci I$, $I\in \cL$. Since $\La 1\ci I \Ra\ci J =1$ for $J\subset I$,  the boundedness of the operator $A_\alpha$ implies
\[
\left\| \sum_{J\in\cL:J\subset I} \alpha\ci J \1\ci J \right\|^p_p \le C\|\1\ci I \|_p^p = C |I|,
\]
which means $\alpha \in \gb_\infty^{q, (p)} = \gb_\infty^q$ (by Theorem \ref{t2.3} the space $\gb_\infty^{q, (p)}$ does not depend on $p$).

Let us now prove sufficiency. Let
\[
E_k=\{x\in \X: | M_\cL f (x)| > 2^k\}, \qquad \text{and}\qquad \cE_k = \{I\in \cL: I \subset E_k\}.
\]

We can write
\[
A_\alpha f = \sum_{k\in\Z} (A_\alpha f)\ci{\cE_k \setminus \cE_{k+1}},
\]
where $(A_\alpha f)\ci{\cE_k \setminus \cE_{k+1}}$ denotes the coordinate projection \eqref{coord.proj} of $A_\alpha f$ with $\cE= \cE_k\setminus\cE_{k+1}$.  

Assume that $\|\alpha\|_{\gb_p^q} \le 1$, and let $\alpha\ci{\cE_k \setminus \cE_{k+1}}$ be the corresponding coordinate projection \eqref{coord.proj} of $\alpha$.
Since $|\E\ci I f| \le |M\ci \cL f| \le 2^{k+1}$ on $I\in\cE_k \setminus \cE_{k+1}$, we conclude that
\[
\left\| (A_\alpha f)\ci{\cE_k \setminus \cE_{k+1}} \right\|_{L^p(\ell^q)}^p \le 2^{(k+1)p} \|\alpha\ci{\cE_k \setminus \cE_{k+1}} \|_{L^p(\ell^q)}^p \le
2^{k+1} |E_k|.
\]
Therefore,
\begin{equation}
\label{2.6}
\sum_{k\in\Z} \left\| (A_\alpha f)\ci{\cE_k \setminus \cE_{k+1}} \right\|_{L^p(\ell^q)}^p \le \sum 2^{(k+1)p} |E_k| \le C \|M\ci\cL f\|_p^p \le C_1 \|f\|_{\wt H^p}^p
\end{equation}
which is exactly what we need if $p=q$.

The case $p<q$ is also easy. We can write (in the functional representation)
\begin{align*}
\|A_\alpha f\|\ci{L^p(\ell^q)}^p
&= \int_\X \biggl( \sum_{k\in\Z} \left\| (A_\alpha f)\ci{\cE_k \setminus \cE_{k+1}}(x, \fdot)  \right\|_{\ell^q}^q \biggr)^{\frac1q p} dx
\\
& \le
\int_\X \biggl( \sum_{k\in\Z} \left\| (A_\alpha f)\ci{\cE_k \setminus \cE_{k+1}}(x, \fdot)  \right\|_{\ell^q}^p \biggr)^{\frac1p p} dx  &\qquad & \text{because } \|s\|\ci{\ell^q} \le \|s\|\ci{\ell^p} \ \text{for } p<q
\\ & = \sum_{k\in\Z} \left\| (A_\alpha f)\ci{\cE_k \setminus \cE_{k+1}}  \right\|_{L^p(\ell^q)}^p \le C \|f\|_{\wt H^p}^2  &&  \text{by \eqref{2.6}}
\end{align*}

The case $p> q$ is a bit more complicated.  To treat this case let us first make few simplifications. Of course, without loss of generality we can assume that $f\ge0$ and that all $\alpha\ci I \ge0$.

\begin{lm}
\label{l2.5}
Let $f\in L^p$, $1<p<\infty$, $f\ge 0$. There exists a function $\wt f$, $\wt f \ge f$, $\|\wt f\|_p \le C_p \|f\|_p$ such that
\[
M\wt f \le C'_p \wt f ,
\]
where $M$ is the Hardy--Littlewood maximal function.  
\end{lm}
The condition $Mf \le C f$ for $f\ge0$ is the so-called Muckenhoupt $(A_1)$ condition. It implies, in particular, that for any interval $I$
\[
\fint_I f  \le \frac1C \min_{x\in I} f(x)
\]
\begin{proof}[Proof of Lemma \ref{l2.5}]
Define $\wt f$
\[
\wt f := \sum_{k\ge 0} \gamma^k M^k \wt f ,
\]
where $M^k$ is $k$th iteration of $M$, and $\gamma>0$ is sufficiently small, so that
\[
\gamma \|M h \|_p \le \frac12 \|h\|_p \qquad \forall h\in L^p.
\]
\end{proof}
Replacing $f$ by $\wt f $ from Lemma \ref{l2.5} we can assume without loss of generality that $Mf \le C f$.

It is an easy exercise with H\"{o}lder inequality and Resonance Lemma (the fact that equality in H\"{o}lder inequality is attained) to see that if $p>q$ and $1/p + 1/r =1/q$, then
\[
\|F\|_p = \sup \{ \| Fg\|_q : g\in L^r, \|g\|_r \le 1\}.
\]

Take $g\in L^r$, $g\ge 0$, $\|g\|_r \le 1$. Since clearly $r>q$, applying Lemma \ref{l2.5} to $g^q$ with the exponent $s=r/q$ for $p$, we get a function $\wt g\ge g$ such that $M(\wt g^q) \le C \wt g^q$ and
\[
\|\wt g\|_r^r = \|\wt g^q \|_s^s \le C\| g^q\|_s^s  = C \|g\|_r^r.
\]
So, replacing $g$ by $\wt g$ we can assume without loss of generality that $Mg\le Cg$ and $\|g\|_r\le C$.

To complete the proof, let us first notice that
\begin{equation}
\label{2.7}
\|A_\alpha(fg) \|\ci{L^q(\ell^q)} \le C\|fg\|_q \le C \|f\|_p \|g\|_r \le C' \|f\|_p.
\end{equation}
This inequality follows from the case $p=q$ we discussed above. We used here the fact that $\|\alpha\|\ci{\gb_\infty^{q, (q)}} \le
\|\alpha\|\ci{\gb_\infty^{q, (p)}}$, which follows immediately from H\"{o}lder inequality; note that we do not need here the full equivalence of $\gb_\infty^{q,(p)}$-norms for all $p$.

So, in light of \eqref{2.7}, we only need to show that
\begin{equation*}
 \|(A_\alpha f)g \|\ci{L^q(\ell^q)} \le C \|A_\alpha(fg) \|\ci{L^q(\ell^q)}
\end{equation*}
which follows immediately if  the estimate
\begin{equation*}
 \fint_I |\La f \Ra\ci I g |^q \le C \La fg\Ra\ci I^q.
\end{equation*}
holds uniformly for all $I\in\cL$.

We know that
\begin{equation}
\label{2.8}
\min_{x\in I} f(x) \le \La f \Ra\ci I \le C \min_{x\in I} f(x)
\end{equation}
and that
\begin{equation}
\label{2.9}
\La g^q \Ra\ci I \le C \min_{x\in I} g^q(x) = C (\min_{x\in I} g(x))^q \le C \La g\Ra\ci I^q
\end{equation}
and therefore
\begin{align*}
\fint_I |\La f \Ra\ci I g |^q & =  \La f \Ra\ci I^q \La g^q \Ra\ci I  &&  \\
& \le  C \La f \Ra\ci I^q \La g \Ra\ci I^q                   &  & \text{by \eqref{2.9}}\\
&\le C   \min_{x\in I} f(x)^q  \La g \Ra\ci I^q && \text{by \eqref{2.8}}\\
& \le C  \La fg\Ra\ci I^q
\end{align*}

\end{proof}

%
%
%
%

\section{\texorpdfstring{$\cH^p_q$}{H**p<sub>q} and \texorpdfstring{$\BMOs_q$}{BMO<sub>q} spaces}

Most of result of this section are well known, and are presented there only for the convenience of the reader. However, I believe some proofs are new.


\subsection{\texorpdfstring{$\cH^p_q$}{H**p<sub>q} spaces}
\label{s.Hp}
As we had discussed in Section \ref{MartHardySp}, $L^p$ norm of a function $f$, $p\in (1, \infty)$ ($H^1$ norm if $p=1$) is equivalent to $L^p$ norm of the square function $Sf$.

Acting  by analogy, one can use ``$q$-function'' instead, and consider the space $\wt\cH^p_q$, $p,q\in [1, \infty]$ of formal martingale difference decompositions, such that
\begin{equation}
\label{Hp-norm-1}
\| f\|_{\wt\cH^p_q} := \Bigl\| \Bigl( \sum_{I\in \cL} |\Delta\ci I f  |^q + \sum_{I\in\Aif} |\E\ci I f|^q \Bigr)^{1/q} \Bigr\|_{L^p} <\infty.
\end{equation}
We assume here that the ``martingale differences'' $\Delta\ci I f$ are simply some functions $h\ci I$, constant on children of $I$ and such that $\int_\X h\ci I dx =0$. The functions $E\ci I f$ are just some multiples of $\1\ci I$.  While we do not assume that all $h\ci I $ are martingale differences  for some function $f$,
we will still use notation $\Delta\ci f$, meaning by $f$ the whole collection of such ``martingale differences''.%
\footnote{Such notation is partially justified by the fact, that in the essential case when our collection has only finitely many non-zero terms, all $\Delta\ci f$ are martingale differences of the function $\sum_{I\in\cL} \Delta\ci I f$}

We can also consider the space ${\cH^p_q}$, consisting of formal martingale difference decompositions for which $\E\ci I =0$ for all $I\in \Aif$.

The spaces  $\wt\cH^p_q$and ${\cH^p_q}$  are clearly Banach spaces as closed subspaces of $L^p(\ell^q)$.  

As we discussed above in Section \ref{MartHardySp}, if $f\in \wt\cH^p_2$, then  the series converges to a function in $L^p$ (to a function in $\wt H^1$ if $p=1$), and for $p\in (1, \infty)$ the norm $\|f\|\ci{\wt\cH^p_2}$ is equivalent to the standard $L^p$ norm.

Since $\|x\|_{\ell^2} \le \|x\|_{\ell^q}$ for $q\in [1, 2]$, any formal martingale decomposition $f\in \cH^p_q$, $q\in [1, 2]$ converges to a function in $L^p$. Thus we can identify in this case the spaces $\wt \cH^p_q$ and ${\cH^p_q}$ with function spaces, which we denote $\wt H^p_q$ and ${H^p_q}$ respectively.

For $q>2$ convergence is not clear, so in this case we only consider the spaces $\wt\cH^p_q$ and ${\cH^p_q}$ of formal martingale differences.

\begin{rem}
\label{r3.1}
Informally, we can say that $f\in {\cH^p_q}$ if $\{\Delta\ci I f\}\ci{I\in\cL} \in \gb_p^q$. We are saying ``informally'' here, because $\Delta\ci f$ are not numbers but functions, so we have to interpret the sequence $\{\Delta\ci I f\}\ci{I\in\cL}$ as the sequence of numbers.

In this paper we will interpret this by saying  that each $\Delta\ci I f$ defines entries $x\ci J $, $J\in \chld(I)$, where $x\ci J$ is simply the value of $\Delta\ci I f$ on $J$.
\end{rem}

The space $\cH^1_2$ is the classical martingale  $H^1$ space, and we often will skip index $q$ dealing with the case $q=2$. Spaces $\cH^p_q$ we studied by many authors, for example they were the spaces $H_p^{S_q}$ considered in \cite{Weisz-MartHardy-Studia-1995} ($S$ in the superscript stands for ``square function'').

\begin{rem*}
There are alternative ways of obtaining  entries  $x\ci I$ from the martingale differences $\Delta\ci I$. For example,  one  puts $x\ci I := \left(\E\ci I |\Delta\ci I f|^q\right)^{1/q}$, the condition   $\{ x\ci I\}_{I\in\cL} \in \gb_p^q$, define a martingale Hardy space what is denoted $H_p^{s_q}$ in \cite{Weisz-MartHardy-Studia-1995} (note that the supescript here is $s_q$, unlike $S_q$ in the previous paragraph.

For the lattices of homogeneous type it is not hard to show that this definition is equivalent to the first one. It is also well known and will be seen from what follows, that in the general, non-homogeneous case, the spaces can be different.
\end{rem*}

\subsection{BMO spaces and \texorpdfstring{$H^1$}{H**1}-BMO duality}
We want to define BMO  spaces, so we have the $H^1$-BMO duality, as usual.
\begin{df*}
We say that a formal martingale difference decomposition $f=\sum_{I\in\cL}\Delta\ci I f$ belongs to the space $\BMOs_q$, $q\in (1, \infty)$  
 if the sequence $\{x\ci I\}\ci{I\in\cL}$, obtained from $\{\Delta\ci I f \}\ci{I\in\cL}$ as in Remark \ref{r3.1}, belongs to $\gb_\infty^q$.

The extended BMO spaces $\BMOs_q^\sim$ are obtained by adding to the formal $f\in \BMOs_q$ additional terms
\[
\sum_{I\in \Aif} \E\ci I f, \qquad \| \E\ci I f \|_\infty \le C <\infty \quad \forall I \in \Aif.
\]

\end{df*}

We can rewrite the definition of $\BMOs_q$ by picking $r\in[1, \infty)$ (recall, than defining space $\gb_\infty^{q}$ we first defined the spaces $\gb_\infty^{q, (r)}$ and and then had shown that all $\gb_\infty^{q, (r)}$-norms are equivalent) and saying that $f\in\BMOs_q$ if for any $I\in \cL$
\begin{equation}
\label{3.1}
\fint_I \Bigl( \sum_{J\in\cL: J \subset I} |\Delta\ci{\!J} f|^q \Bigr)^{\frac1q r} dx \le C <\infty
\end{equation}
(uniformly in  $I$), and, in addition
\begin{equation}
\label{3.1.a}
\sup_{I\in\cL} \|\Delta\ci I f \|_\infty <\infty.
\end{equation}

\begin{rem*}
One would expect, that the  condition \eqref{3.1} alone defines the space $\BMOs_q$, but it was known for a long time, that the additional condition is needed. One can look, for example at the 1973 Garsia's book \cite{Garsia1973-Mart_Ineq} where the  $\BMOs_2$ space 
was defined. One can easily see that the definition from  \cite{Garsia1973-Mart_Ineq} is equivalent to the one presented here.    
\end{rem*}

If $|\X|<\infty$, $\BMOs_q\subset \cH^r_q$, so (see Section \ref{s.Hp} above) for $q\in(1, 2]$ the formal martingale difference decomposition $f\in \BMOs_q$ converges to a function in $L^r$. So in the case $|\X|<\infty$ one can identify for $q\in [1, 2]$ the spaces $\BMOs_q$ and $\BMOs_q^\sim$ with function spaces, which we will call $\BMO_q$ and $BMO_q^\sim$ respectively.

The following theorem is known,
but the proof presented here is probably new.

\begin{thm}
The dual of the space $\cH^1_q$ (resp.~${\wt \cH^1_q}$), $1<q<\infty$ is the space $\BMOs_{q'}$ (resp.~$\BMOs^\sim_{q'}$).
\end{thm}

\begin{proof}

We will prove the duality between ${\cH^1_q}$ and $\BMOs_{q'}$, the duality between $\wt\cH^1_q$ and $\BMOs_{q'}^\sim$ follows trivially.

The sufficiency of the condition $g\in \BMOs_{q'}$ for the boundedness of the linear functional $f\mapsto \La f, g\Ra$ on ${\cH^1_q}$ follows immediately from Theorem \ref{t2.3}.

To prove the necessity of this condition, let us note, that by the definition ${\cH^q_1}$ can be identified with a subspace of $\gb_1^q$ (by identifying the family $\{\Delta\ci I g\}\ci{I\in\cL}$ with an element in $\gb_1^q$ as described in Remark \ref{r3.1}).

Thus a linear functional $\f$ on ${\cH^1_q}$ can be extended by Hahn--Banach theorem to a functional on $\gb_1^q$, so by Theorem  \ref{t2.3} there exists $\wt g = \{ \wt g\ci I \}\ci{I\in\cL} \in \gb_\infty^{q'}$,   $\|\wt g \|\ci{\gb_\infty^{q'}} \le C\|\f\|$ such that the  functional $\f$ is given by
\begin{equation}
\label{3.3}
\f(f)
=\int_\X \sum_{I\in\cL} \sum_{J\in\chld(I)} \Delta\ci I f (x) \wt g\ci J dx  
=\sum_{I\in\cL} \,\,\int_\X \sum_{J\in\chld(I)} \Delta\ci I f (x) \wt g\ci J dx .
\end{equation}
We would like to interpret the function $\sum_{J\in\chld(I)} \wt g\ci J \1\ci J$ as a martingale difference,
but this function does not have zero average.  But since $\int_\X \Delta\ci I f dx = 0$, the integrals in the right side of \eqref{3.3} do not change if we subtract from $\wt g\ci J$, $J\in\chld(I)$ a constant $c=c\ci I$.

Therefore, if for $J\in\chld(I)$ we define $g\ci J := \wt g\ci J - |I|^{-1} \sum_{I'\in\chld(I)} \wt g\ci{I'} |I'|$, we get that
\[
\f(f) =  \sum_{I\in\cL} \,\,\int_\X \sum_{J\in\chld(I)} \Delta\ci I f (x)  g\ci J dx  .
\]
But now  the functions $\sum_{J\in\chld(I)}  g\ci J \1\ci J$ have zero average, so we can treat them as martingale differences.

Let us check that $\{ g\ci I \}\ci{I\in\cL} \in \gb_\infty^{q'}$. Using the fact that the averaging operator $f\mapsto \La f\Ra\ci I \1\ci I$ is a contraction in all $L^p$, $p\in[1, \infty]$ (it follows immediately from H\"{o}lder inequality), we can see that
\[
\Bigl\| \sum_{J\in\chld(I)} g\ci J \1\ci J \Bigr\|_{q'} \le 2
\Bigl\| \sum_{J\in\chld(I)} \wt g\ci J \1\ci J \Bigr\|_{q'}
\]
Using this inequality we get that for $I_0\in\cL$
\begin{align*}
\sum_{J\in\cL: J\subset I_0} |g\ci J|^{q'}|J|  & = |g\ci{I_0}|^{q'} \cdot |I_0| +
\sum_{I\in\cL: I\subset I_0} \Bigl\| \sum_{J\in\chld(I)} g\ci J \1\ci J \Bigr\|_{q'}^{q'}
\\ &
\le |g\ci{I_0}|^{q'} \cdot |I_0| +
2^{q'} \sum_{I\in\cL: I\subset I_0} \Bigl\| \sum_{J\in\chld(I)} \wt g\ci J \1\ci J \Bigr\|_{q'}^{q'}
\\ & \le |g\ci{I_0}|^{q'} \cdot |I_0| + 2^{q'} |I_0| \cdot \| \wt g \|^{q'}_{\gb_\infty^{q',(q')}}.
\end{align*}

Noticing that $|\wt g\ci{I_0}| \le \| \wt g \|_{\gb_\infty^{q',(q')}}$, and therefore $| g\ci{I_0}| \le 2\| \wt g \|_{\gb_\infty^{q',(q')}}$, and taking into account that  $\gb_\infty^{q',(r)}$ norms are equivalent for all $r\in[1, \infty)$, we conclude that $\|\{ g\ci I\}\ci{I\in\cL} \|_{\gb_\infty^{q'} }\le  C \|\wt g\|_{\gb_\infty^{q'} }$.   

Thus $g =\{g\ci I \}\ci{I\in\cL} \in \BMOs_{q'}$.
\end{proof}

\subsection{\texorpdfstring{$\BMOs_q$}{BMO<sub>q} as function spaces}

\begin{prop}
\label{p.MBO-FS}
For $q\in [1, 2]$ the space $\BMOs_q$ can be identified with a function space, i.e.~for each formal martingale decomposition $f=\sum_{I\in\cL} \Delta\ci I f \in\BMOs_q$ there exists a locally integrable function $\tilde f$ such that for all $I\in\cL$
\[
\Delta\ci I f = \Delta\ci I \tilde f.
\]
A similar statement holds for spaces $\BMOs_q^\sim$ as well.
\end{prop}

\begin{rem*}
As it can be seen from a simple example below, the martingale difference decomposition $f = \sum_{I\in\cL} \Delta\ci I f \in \BMO_q$, $q\in (1, 2]$ does not necessarily converge if $|\X|=\infty$.  

Nemely, let $\cL$ be the standard dyadic lattice $\cD$ in $\R$, and let $I_k = [0, 2^k)$, $k\in\N$. Consider the formal martingale sum $f=\sum_{k=1}^\infty \Delta\ci{I_k} f$, where
\[
\Delta\ci{I_k} f = \1_{[0, 2^{k-1})} - \1_{[ 2^{k-1}, 2^k)} .
\]
It is easy to see that $f\in\BMOs_q$ for all $q\in (1, \infty)$, but the series clearly diverges.
\end{rem*}

\begin{proof}[Proof of Proposition \ref{p.MBO-FS}]
It is sufficient to analyze the convergence on each interval $J\in \mathfrak A_{-\infty}^0$ separately.  

If $|J|<\infty$, then the series $f\ci J := \sum_{I\in\cL: I \subset J} \Delta\ci I f$ belongs to $\cH_q^2$.  
Therefore, as it was discussed before in Section \ref{s.Hp}, the series converges to an $L^2$ function.

Let now consider $J\in \mathfrak A_{-\infty}^0$ such that $|J| = \infty$. It is not hard to see that any such interval can be represented as the union
\[
J = \bigcup_{k\ge 1} I_k, \qquad I_k\in \chld(I_{k+1})\ \forall k \ge 1
\]
(note, that here $k$ is \emph{not} the number of generation).

Let
\[
\cL(J):= \{ I\in \cL: I\subset J , I\ne I_k \forall k \in \N\}
\]
 so the collection $\{I\in\cL :I\subset J\} $ is split into a disjoint union of $\cL$ and the set $\{I_k, :k\in\N\}$.

For $k=2, 3, \ldots $ let $\alpha_k$ be the value of $\Delta\ci{I_k}$ on $I_{k-1}$, ald let $\alpha_0 =0$. Define the function $\tilde f$ on $J$ by
\[
\tilde f := \sum_{I\in\cL(J)} \Delta\ci I f + \sum_{s\in\N} \Bigl( \Delta\ci{I_s} f - \alpha_s \1\ci{J} \Bigr)
\]
Let us show that the sum restricted to any of the above  intervals $I_k$ converges in $L^2(I_k)$. This will immediately imply that   $\Delta\ci I \tilde f = \Delta\ci I f$ for all $I\in\cL$,  $I\subset J$.  
 
The second sum trivially converges, because $\Delta\ci{I_s} f - \alpha_s \1\ci{J} =0$ on $I_k$ if  $s>k$.

Let us show the convergence of the first sum. Note, that we only need to count the terms $\Delta\ci I f $ with $I\in\cL(J)$, $I\subset I_k$, because the terms with $I\in\cL(J)$, $I\not\subset I_k$ are zero on $I_k$.

%

Condition $f\in \BMOs_q$ implies that
\[
\sum_{I\in\cL(J): I\subset I_k} \Delta\ci I f \in \cH^2_q \subset \cH^2_2,
\]
so the sum converges in $L^2(I_k)$.
\end{proof}

\section{ \texorpdfstring{$L^p$}{L**p} bounds of paraproducts
}

\subsection{Martingale differences do not form a strong unconditional basis in  \texorpdfstring{$H^p$}{H**p} in the non-homo\-gen\-eous case}

In \cite{NaTr-AA} the notion of a strong unconditional basis was introduced. A system of nontrivial  subspaces $\cE_j$ (of a Banach space $X$), $j\in \cJ$ (where $\cJ$ is a some countable set) was called a \emph{strong unconditional basis} if 
\begin{enumerate}
	\item The  linear span  $\mathcal{L}\{\E_{j}\,:\,{j\in\cJ}\}$ is dense 
in $X$;

	\item There exists an ideal Banach space $Y$ of sequences 
$\{c_{j}\}\ci {j\in\cJ}$ and a constant $A>0$ such
that for any sequence $\{x_{j}\}\ci {j\in\cJ}$, $x_{j}\in \cE_{j}$ with
finitely many non-zero elements
\[
\frac1A \Bigl\| \sum_{j\in \cJ} x_{j} \Bigr\|\ci {X} \le
\Bigr\| \{\| x_{j} \|\}\ci {j\in\cJ} \Bigl\|\ci {Y} \le A
\Bigl\| \sum_{j\in \cJ} x_{j} \Bigr\|\ci {X}\,.
\]
\end{enumerate}
Recall that a Banach space $Y$ of 
sequences $\{ c_{j}\}\ci{j\in\cJ}$ of complex numbers is called {\em ideal}
if for any sequence of factors $\alpha_{j}$, $|\alpha_{j}|\le 1$ the 
sequence $\{ \alpha_{j}c_{j}\}\ci {j\in\cJ}\in Y$ and
$\| \{ \alpha_{j}c_{j}\}\ci {j\in\cJ} \|\ci{Y} \le \| \{ c_{j}\}\ci {j\in\cJ}
\|\ci{Y}$ 

Note, that a strong unconditional basis is an unconditional basis, meaning that any vector $x\in X$ admits a unique representation 
	\[
	x=\sum_{j\in\cJ} x_j, \qquad x_j\in\cE_j, 
	\]
and the series converges unconditionally, i.e.~independently of the ordering of 
$\cJ$. 
	
One can easily see that the martingale difference spaces $D\ci I = \Delta\ci I$ form an unconditional basis	in $H^p$, $p\in[1, \infty)$. It is also well known that for a homogeneous lattice $\cL$ the subspaces $D\ci I$ form a strong unconditional basis. 

Unfortunately, as we demonstrate below, that is not the case in the general situation. 
	
If the system of the martingale difference spaces $D\ci I$ were a strong unconditional basis, one could guess that the natural ``coefficient space''  for $H^p$ should be the Triebel--Lizorkin type space $\gb^2_p$.

In other words, one could guess that 
one could get an equivalent norm in $H^p$ by replacing the functions $\Delta\ci I f$ in the square function by multiples of $\1\ci I$. The norms have to be equivalent on singletons $f =\Delta\ci I f$, so if one wants replace functions $\Delta\ci I f $ by $c\ci I \1\ci I$,  $c\ci I = c\ci I(f)$, the functions $\Delta\ci I f $ and $c\ci I \1\ci I$ should have equivalent $L^p$ norms (uniformly in $I$).

If everything works when the norms of $\Delta\ci I f $ and $c\ci I \1\ci I$ are equivalent, it works when they are equal. So everything reduces to the question on whether the quantity
\begin{equation}
\label{p-norm-1}\Bigl\| \Bigl( \sum_{I\in\cL}\bigl( \E\ci I |\Delta\ci I f |^p \bigr)^{2/p}  \Bigr)^{1/2} \Bigr\|_p
\end{equation}
gives an equivalent norm on $H^p$.   

The answer is well known to be ``yes'' in the case when the lattice if of homogeneous type. In fact, in this case for $q\in [1, \infty)$ the  averages
$(\E\ci I |\Delta\ci I|^q )^{1/q}$ are equivalent, so one can replace $\Delta\ci I f$ by any of these averages (the case $q=2$ is usually considered in the literature).

In the general case, as the theorem below asserts, only ``half'' of necessary inequalities is holds, so the answer is unfortunately ``no''. 

Note, that Theorem \ref{Lp-av} does not imply that the system of martingale difference spaces $D\ci I$ is not a strong unconditional basis: it only implies that a particular norm on coefficient space does not give an equivalent norm. However, modifying the proof of Theorem \ref{Lp-av} one can show that indeed the martingale difference spaces $D\ci I$ do not form a strong unconditional basis in $H^p$.     

\begin{thm} 
\label{Lp-av}
Let $f\in H^p$.   
\begin{enumerate}
\item For $p\in[1, 2]$ the inequality
\begin{equation}
\label{p-norm-2}
\Bigl\| \Bigl( \sum_{I\in\cL}\bigl( \E\ci I |\Delta\ci I f |^p \bigr)^{2/p}  \Bigr)^{1/2} \Bigr\|_p
 \le
C  
\Bigl\| \Bigl( \sum_{I\in\cL} |\Delta\ci I f |^2  \Bigr)^{1/2} \Bigr\|_p
\end{equation}
holds; here $C=C_p$ and does not depend on $f$ and $\cL$.

\item For $p\in [2, \infty)$ the opposite inequality
\begin{equation}
\label{p-norm-3}
\Bigl\| \Bigl( \sum_{I\in\cL} |\Delta\ci I f |^2  \Bigr)^{1/2} \Bigr\|_p
 \le
C  
\Bigl\| \Bigl( \sum_{I\in\cL}\bigl( \E\ci I |\Delta\ci I f |^p \bigr)^{2/p}  \Bigr)^{1/2} \Bigr\|_p
\end{equation}
holds with $C=C_p$.
\item For $p\in(2, \infty)$ the inequality \eqref{p-norm-2} fails, i.e.~for each $p>2$ one can find a lattice $\cL$ and $f\in H^p$ such that the left side of \eqref{p-norm-2} is infinite (while the right side is finite because $f\in H^p$).    

\item  For $p\in [1, 2)$ the inequality \eqref{p-norm-3} fails (in the same sense an in the statement \textup{\cond3})

\end{enumerate}
\end{thm}

\subsubsection{Proof of two estimates in Theorem \ref{Lp-av}   }
 To prove statement \cond1, let us consider the sequence $\{|\Delta\ci I f|^p\}\ci{I\in\cL}\in \gb_1^q$, $q=2/p$, where as in Remark \ref{r3.1} $|\Delta\ci I f|^p$ defines entries $x\ci J$, $J\in\chld(I)$, $x\ci J$ being the value of $|\Delta\ci J f|^p$ on $J$.  Then the estimate \eqref{p-norm-2} follows from immediately from the boundedness of the averaging operator $\Av$ in $\gb_1^q$, $q=2/p$
\[
(\Av x)\ci I = |I|^{-1} \sum_{J\in\chld(I)} x\ci J |J|.
\]
To prove that $\Av$ is bounded, let us notice that its adjoint  $\Av^*$ is the ``forward shift''
\[
(\Av^* x)\ci I = x\ci{\wt I}, \qquad \wt I \text{ is a parent of } I.   
\]

We want to show that this operator is bounded in $\gb_\infty^{q'} = (\gb_1^q)^*$. If $x=\{x\ci I \}\ci{I\in\cL} \in \gb_\infty^q$, then for $J\in\cL$
\[
 \sum_{I\in\cL:I \subset J} | (\Av^* x)\ci I |^{q'} \1\ci I   =
 | x\ci{\wt J} |^{q'} 1\ci J + \sum_{I\in\cL:I \subset J} | x\ci I |^{q'} \1\ci I     
\]
where $\wt J$ is the ``parent'' of $J$. Since $|x\ci{\wt J}|\le \|x\|_{\gb_\infty^{q'}}$ and
\[
\fint_J \sum_{I\in\cL:I \subset J} | x\ci I |^{q'} \1\ci I  \le \| x\|_{\gb_\infty^{q'}}^{q'}
\]
we conclude that
\[
\fint_J \Bigl( \sum_{I\in\cL:I \subset J} | (\Av^* x)\ci I |^{q'} \1\ci I  \Bigr) dx \le 2 \| x\|_{\gb_\infty^{q'}}^{q'},
\]
which proves that $\Av^*$ is bounded. Therefore $\Av$ is a bounded operator in $\gb_1^q$, which proves \eqref{p-norm-2}.

Statement \cond2 follows from \eqref{p-norm-2} by duality. Namely, take $g\in H^{p'}$, $1/p+1/p' = 1$, $\|g\|_{H^{p'}}\le 1$ and estimate
\begin{align*}
\Bigl|\int_\X fg dx \Bigr|  & \le  \sum_{I\in\cL} \int_I |\Delta\ci I f \Delta\ci I g| dx
\\
& \le \sum_{I\in\cL} \int_I \bigl(\E\ci I| \Delta\ci I f|^p\bigr)^{1/p}
\bigl(\E\ci I| \Delta\ci I g|^{p'}\bigr)^{1/p'} dx \\
& =
 \int \sum_{I\in\cL} \bigl(\E\ci I| \Delta\ci I f|^p\bigr)^{1/p}
\bigl(\E\ci I| \Delta\ci I g|^{p'}\bigr)^{1/p'} dx
\\
& \le
\int \Bigl( \sum_{I\in\cL}\bigl(\E\ci I| \Delta\ci I f|^p\bigr)^{2/p} \Bigr)^{1/2}
\Bigl( \sum_{I\in\cL} \bigl(\E\ci I| \Delta\ci I g|^{p'}\bigr)^{2/p'} \Bigr)^{1/2}
\\
& \le
\Bigl\| \Bigl( \sum_{I\in\cL}\bigl( \E\ci I |\Delta\ci I f |^p \bigr)^{2/p}  \Bigr)^{1/2} \Bigr\|_p
\Bigl\| \Bigl( \sum_{I\in\cL}\bigl( \E\ci I |\Delta\ci I g |^{p'} \bigr)^{2/p'}  \Bigr)^{1/2} \Bigr\|_{p'}
\end{align*}
By \eqref{p-norm-2}
\[
\Bigl\| \Bigl( \sum_{I\in\cL}\bigl( \E\ci I |\Delta\ci I g |^{p'} \bigr)^{2/p'} \1\ci I \Bigr)^{1/2} \Bigr\|_{p'}
\le C \|g\|_{H^{p'}}
\]
So by taking supremum over $g\in H^{p'}$, $\|g\|_{H^{p'}}\le 1$ and taking into account that the dual of $H^p$ is isomorphic to $H^{p'}$, we get
\[
\|f\|_{H^p} \le C \Bigl\| \Bigl( \sum_{I\in\cL}\bigl( \E\ci I |\Delta\ci I f |^p \bigr)^{2/p}  \Bigr)^{1/2} \Bigr\|_p, 
\]
which is exactly condition \cond2. \hfill \qed

\subsubsection{Counterexamples in Theorem \ref{Lp-av}} 
\label{s.c-ex:averaging}
To prove \cond3, take $I_0= [0, 2)$. Fix $n\in\N$, $n>2$ and let 
\[
I_k = [0, r^k), \qquad J_k = [r^k, r^{k-1}), 
\]
where $r=1-1/n$,  $k=1, 2, \ldots n$. Note that $I_{k-1}$ is a disjoint union of $I_k$ and $J_k$.

We assume here that $I_k, J_k\in\cL_k$; we will only consider functions whose only non-zero martingale differences are $\Delta\ci{I_{k-1}} f$, $k=1, 2, \ldots, n$, so the other intervals in $\cL_k$ are irrelevant for our construction.

For $k=1, 2, \ldots, n$ define
\[
\Delta\ci{I_{k-1}} f = \1\ci{J_k} - \alpha \1\ci{I_k}
\]
where
$
\alpha= 1/(n-1)$, so $\int_\X \Delta\ci{I_{k-1}} f dx =0$. 


We can estimate that on $I_0$
\begin{align*}
\Bigl( \sum_{k=0}^{n-1} |\Delta\ci{I_k} f |^2  \Bigr)^{1/2} \le \Bigl( 1 + \sum_{k=1}^n \alpha^2 \Bigr)^{1/2} \le \Bigl( 1 + n \frac{1}{(n-1)^2} \Bigr)^{1/2} \le 2^{1/2}
\end{align*}
(each point $x\in I_0$ belongs to at most one of the intervals $J_k$, which  contributes $1$ to the sum, and each $I_k$ contributes $\alpha^2$). Therefore
\begin{equation}
\label{p-norm-5}
\Bigl\|  \Bigl( \sum_{k=0}^{n-1} |\Delta\ci{I_k} f |^2  \Bigr)^{1/2} \Bigr\|_p \le 2^{1/2} .
\end{equation}

On the other hand for $x\in I_{k-1}$
\[
\E\ci{I_{k-1}} |\Delta\ci{I_{k-1}} f |^p = \frac1n  + (1-1/n) \alpha^p  \ge \frac1n
\]
so for $x\in I_n$
\[
\sum_{k=1}^{n} \bigl(\E\ci{I_{k-1}} |\Delta\ci{I_{k-1}} f |^p \bigr)^{2/p} \ge  n \left(\frac1{n}\right)^{2/p}
=  n^{1-2/p}.
\]
Since $p>2$ we have $n^{1-2/p}\to \infty$ as $n\to \infty$, so by increasing $n$ we can make the left side of \eqref{p-norm-2} as large as we want (because $|I_n|=(1-1/n)^n> 1/2e$ for sufficiently large $n$). But by \eqref{p-norm-5} the right side of \eqref{p-norm-2} is uniformly bounded. Thus, the uniform (in all lattices) estimate \eqref{p-norm-2} fails.

Repeating the construction (with $n\to \infty$) on disjoint intervals, we get a lattice where the uniform (in $f$) estimate \eqref{p-norm-2} fails. But from here one can easily construct a function such that the right side of \eqref{p-norm-2} is finite, but the left side is infinite.

The same construction allows us to prove statement \cond4 as well. Namely, we can easily see that on $J=\cup_{k=1}^n J_k$
\[
\Bigl( \sum_{k=0}^{n-1} |\Delta\ci{I_k} f |^2  \Bigr)^{1/2} \ge 1. 
\]
Note that $|J|=1-r^n = 1-(1/n)^n$, so  for sufficiently large $n$, we can estimate that $|J|>1/2$. Therefore
\[
\Bigl\|  \Bigl( \sum_{k=0}^{n-1} |\Delta\ci{I_k} f |^2  \Bigr)^{1/2} \Bigr\|_p \ge 2^{-1/p} .
\]
On the other hand for $x\in I_{k-1}$
\[
\E\ci{I_{k-1}} |\Delta\ci{I_{k-1}} f |^p = \frac1n  + (1-1/n) \alpha^p =   \frac1n + \left( \frac{n-1}{n} \right) \left(\frac1{n-1}\right)^p \le \frac2n
\]
so for $x\in I_0$
\[
\sum_{k=1}^{n} \bigl(\E\ci{I_{k-1}} |\Delta\ci{I_{k-1}} f |^p \bigr)^{2/p} \le  n \left(\frac2{n}\right)^{2/p}
=  2^{2/p} n^{1-2/p}.
\]
Therefore
\[
\Bigl\| \Bigl( \sum_{k=0}^{n-1}\bigl( \E\ci{I_k} |\Delta\ci {I_k} f |^p \bigr)^{2/p}  \Bigr)^{1/2} \Bigr\|_p \le 2^{1/p} n^{1/2-1/p} \to 0 \qquad \text{as } n\to \infty, 
\]
because $p<2$.

So, for $p\in [1, 2)$ the uniform estimate \eqref{p-norm-3} fails, and from here is is easy to get a function for which the right side is finite, but the left side is infinite.  
\hfill \qed

\subsubsection{Not a strong unconditional basis}
\begin{prop}
There exist a lattice $\cL$ such that the martingale difference spaces $D\ci I$ do not form a strong unconditional basis. 
\end{prop}

This proposition also demonstrates, that unlike the case $p=2$ the uniform boundedness  in $L^p$ of the blocks $T\ci I$ of a martingale transform $T$ does not imply the boundedness of $T$ in $L^p$, $p\ne 2$. 

The proof of the Proposition can be obtained by modifying the construction in Section \ref{s.c-ex:averaging}. Define $I_0 =[0,1)$. Fix $n\in\N$, $n>2$. Let split $I_0$ into two subintervals, $I_1$ and $J_1$, where $|J_1| = (1/n) |I_0|$, so $|I_1| =(1-1/n)|I_0|$, and let us split intervals $I_1$ $J_1$ into two equal subintervals, let us call them, $I_1^k$, $J_1^k$, $k=1, 2$. These four intervals will be children of $I_0$. 

For an interval $I$, let $h\ci I$ be the Haar function (normalized in $L^\infty$), 
\[
h\ci I = \1\ci{I_+} -\1\ci{I_-}, 
\]
where $I_+$ and $I_-$ are the right and left halves of $I$ respectively. 

Define the martingale differences in $D\ci{I_0}$, 
\[
\Delta\ci{I_0} f = h\ci{J_1} + \alpha  h\ci{I_1}, \qquad \Delta\ci{I_0} g = \beta (h\ci{J_1} + h\ci{I_1}), 
\]
where as in Section \ref{s.c-ex:averaging} $\alpha = (n-1)^{-1}$, and $\beta = n^{-p} \left(1+ (n-1)^{1-p}\right)^{1/p}$, so
\begin{equation}
\label{D-f_D-g}
\| \Delta\ci{I_0} f\|_p = \|  \Delta\ci{I_0} g \|_p. 
\end{equation}

We then apply the same construction to the ``children'' $I_1^k$ of $I_0$, then to all ``children'' $I_2^j$ of all $I_1^k$ and so on. Note, that we do not care about the ``children'' of the ``smaller'' intervals $J_r^k$, because we put the martingale differences to be zero for all intervals different from one of $I_k^j$. 

So, we get the collection of intervals $I_k^j$ and the corresponding martingale differences $\Delta\ci{I_k^j} f$, $\Delta\ci{I_k^j} g$, constructed the same way as in \eqref{D-f_D-g}. 

But now let us notice that this construction just models the construction from Section \ref{s.c-ex:averaging}. Namely, for every $k\in\N$, the total length of the intervals $I_k^j$ is exactly the length of the interval $I_k$ from Section \ref{s.c-ex:averaging}. It is easy to see that the function
$
\sum_j |\Delta\ci{I_k^j} f| 
$
and the function $|\Delta\ci{I_k} f|$ from Section \ref{s.c-ex:averaging} have the same distribution function. Moreover, the corresponding square functions
\[
\Bigl(\sum_{k,j} |\Delta\ci{I_k^j} f|^2 \Bigr)^{1/2} \qquad \text{and} 
\qquad 
\Bigl(\sum_{k} |\Delta\ci{I_k} f|^2 \Bigr)^{1/2} 
\]
also have the same distribution function.   

The distribution functions of the square function 
\[
\Bigl( \sum_{k,j} |\Delta\ci{I_k^j} g |^2 \Bigr)^{1/2} 
\]
and of the function 
\[
\Bigl( \sum_{k} \left(\E\ci{I_k}|\Delta\ci{I_k}  |^p \right)^{2/p} \Bigr)^{1/2}
\]
from Section \ref{s.c-ex:averaging} also coincide. Therefore, all estimate from Section \ref{s.c-ex:averaging} apply here, and repeating the reasoning from this section we prove the proposition. \hfill \qed

\subsection{``Paraproduct'' version of embedding theorem} Let $b= \{b^I \}\ci{I\in\cL}$ be a family of functions such that $b^I $ is supported on $I$ and is constant on ``children'' of $I$. Define a ``paraproduct type'' operator $\wt\pi = \wt\pi_b$ by
\[
\wt \pi_b f (\fdot, k) = \sum_{I\in\cL:\rk(I) =k} \La f\Ra\ci I b^I.
\]
If $b^I =\Delta\ci I b$ for some scalar function $b$, this is just the classical paraproduct, so that is where our operator came from. However, we do not assume here orthogonality of $b^I$ to constants, so $b$ here is just a collection of functions $b^I$.

We are interested when this operator is a bounded operator $L^p\to L^p(\ell^q)$ (or from $\wt\cH^p\to L^p(\ell^q)$, if we are interested in he case $p=1$).

If each $b^I$ is constant on $I$, the answer is given by Theorem \ref{t2.4} above, and it does not depend on $p\in [1, \infty)$. In the general case, if we do not assume that the lattice is homogeneous, the answer generally depends on $p$; one can easily come up with a counterexample in the simplest situation when only $b^I$ with $I$ in a disjoint family are non-zero.  

\begin{thm}
\label{t.para-embed}
Let $p\in[1,\infty)$, $q\in (1,\infty)$. 
The operator $\wt\pi_b$ defined above is a bounded operator $\wt\cH^p\to L^p(\ell^q)$  if and only if
\begin{equation}
\label{4.1}
\sup_{I\in\cL} \fint_I \Bigl(\sum_{J\in\cL:J\subset I} | b^J (x) |^q    \Bigr)^{\frac1q p} dx = K^p < \infty
\end{equation}
Moreover, the norm of $\wt\pi_b$ is estimated by $CK$, where $C=C(p)$.
\end{thm}

\begin{proof}[Proof: necessity and the easy case $p\le q$]  
The necessity of the condition \eqref{4.1} is trivial, one just needs to test the boundedness of $\wt \pi_b$ on characteristic function $\1\ci I$, $I\in\cL$ an when computing the $L^p(\ell^q)$-norm only count $\La f\Ra\ci J b^J$ corresponding to $J\subset I$.

To prove the sufficiency, let us first fix the notation. As in the proof of Theorem \eqref{t2.4} let $E_k:= \{x\in\X: M\ci\cL f(x) > 2^k \}$ and let $\cE_k:= \{I\in\cL: I\subset E_k\}$.

Let $(\wt \pi_b f)\ci{\cE_k\setminus \cE_{k+1}}$ be the coordinate projection of $\wt \pi_b f$ corresponding to the set $\cE_k\setminus \cE_{k+1}$ (recall, that by \eqref{coord.proj.1} the coordinate projection can be defined for arbitrary function in $L^p(\ell^q)$, not just for elements of $\gb_p^q$).

The sufficiency for $p\le q$ is proved absolutely the same way as in Theorem \ref{t2.4}: using absolutely the same reasoning as there, we get the analogue of \eqref{2.6}, namely that
\[
\sum_{k\in\Z} \left\| (\wt\pi_b f)\ci{\cE_k\setminus \cE_{k+1}} \right\|_{L^p(\ell^q)}^p \le C\|f\|_p^p ,
\]
(this inequality holds for all $p, q \in (1, \infty)$ with $C=C(p, q)$).

This immediately gives the desired estimate if $p\le q$ the same way it was done in the proof of Theorem \ref{t2.4}: we really did not use the fact that $\alpha\ci I$ were constants there, the estimate works for arbitrary functions.
\end{proof}

\subsubsection{Stopping moments and the hard estimate in Theorem \ref{t.para-embed}} To treat the estimate in the situation when $p>q$ we employ the stopping moment technique. 

First of all let us note that if $p>q>1$, then $p>1$, so the space $\wt\cH^p$ is isomorphic to $L^p$. That means that without loss of generality we can assume $f\ge0$, which we will do in what follows (note, that we cannot do that for $\wt\cH^1$). 

So, let us assume $f\ge 0$, $f\in L^p$. 
Fix some $k_0\in Z$ (we later let $k_0\to -\infty$) and define the first generation $\cG^*_1$ of stopping intervals to be the collection of maximal (by inclusion) intervals 
$I\in \cE_{k_0}$  (the sets $\cE_k$ and $E_k$ were  defined above in the beginning of the proof of Theorem \ref{t.para-embed}).

We then construct the generations $\cG^*_k$ of stopping intervals by induction, by taking for each $J\in \cG^*_k$ some disjoint subintervals $I\in \cL$, $I\subset J$ to get $\cG_{k+1}^*$.     Namely, suppose we have generation $\cG_k^*$ of stopping moments.

Let
\[
G_k = \bigcup_{I\in\cG_k^*} I,
\]
and denote $\cG_k = \{ I\in\cL: I\subset G_k\}$ (so $\cG_k^*$ is the collection of all maximal intervals in $\cG_k$).
Also, for $J\in\cL$ let
\[
\rf(J):=\max\{ k\in \Z : J\subset E_k\} = \max\{k\in\Z: J \in\cE_k\}.
\]

For an  interval $J\in \cG_k^*$ we consider all maximal subintervals $I\in\cL$, $I\subset J\cap E_{r+2}$, where $r= \rf(J)$. The collection of such intervals constructed for all $J\in\cG^*_k$ is the generation $\cG_{k+1}^*$ of stopping moments.

It is easy to establish the following properties of stopping moments.

\begin{lm}
For any $J\in\cG_k^*$

\begin{enumerate}
\item 
$\La f\Ra\ci J \le 2^{\rf(J) + 1}$;

\item $\La f\Ra\ci J \ge 2^{\rf(J) }$;

\item 
For any $I\in \cG_k\setminus \cG_{k+1}$, $I \subset J$ the estimate $\La f\Ra\ci I \le 2^{\rf(J) + 2}$ holds;

\item 
Finally
\[
| J\cap G_{k+1} | \le \frac12 |J|
\]
\end{enumerate}
\end{lm}

\begin{proof}
Property \cond1 holds because by the definition of $\rf(J)$, we have $J\not\subset E_{\rf(J)+1}$. Property \cond2 holds because by the construction $J$ is a maximal subinterval of some $E_j$, and therefore it is a maximal subinterval of $E_{\rf(J)}$ (recall that $\rf(J)\ge j$). But since every $E_r$ can be represented as $E_r = \cup\{J\in\cL: \La f \Ra\ci J > 2^r \}$ and $f\ge0$, the inequality
\begin{equation}
\label{max-Er}
\La f\Ra\ci J \ge 2^r
\end{equation}
 holds for all \emph{maximal} subintervals $J$ of $E_r$.   
Here we used the trivial fact that for a non-negative function the average over a union of sets is at least the minimum of averages over each set. 

Property \cond3 holds because  by the construction of $\cG^*_{k+1}$ for any such $I$ we have $I\not\subset E_{\rf(J)+2}$. 

Finally,  any    $I\in \cG_{k+1}^*$, $I \subset J$  is, by the construction, a maximal subinterval of
$E_{\rf(J)+2}$, so by \eqref{max-Er} with $\rf(J)+2$ instead of $r$, $\La f\Ra\ci I \ge 2^{\rf(J)+2}$.   Since $J\cap G_{k+1}$ is a disjoint union of such intervals $I$, we get using condition \cond1
\[
2^{\rf(J)+1} |J| \ge \int_J f dx \ge \int_{J\cap G_{k+1}} f dx \ge 2^{\rf(J)+2} | J\cap G_{k+1} |
\]
(we used the above estimate $\La f\Ra\ci I \ge 2^{r+2}$ for the last inequality), which gives us property \cond4.
\end{proof}

We are now ready to prove the estimate. For an interval $J\in\cG_k^*$ let
\[
\cG(J):= \{ I\in\cL: I\subset J, I\notin \cG_{k+1}\}
\]
(which is exactly the collection of intervals $I$ from the property \cond3 of generations $\cG^*_k$). Then
\[
(\wt\pi_b f)\ci{\cE_{k_0}} = (\wt\pi_b f)\ci{\cG_{1}} =
\sum_{k=1}^\infty \sum_{J\in \cG^*_k} (\wt\pi_b f)\ci{\cG(J)} ,
\]
and each $(\wt\pi_b f)\ci{\cG(J)}$ is supported on $J$.
So, taking $g\in L^{p'}$, $\|g\|_{p'}\le 1$ we can estimate
\begin{equation}
\label{pbf-g}
\int_{E_{k_0}} \| (\wt\pi_b f)\ci{\cE_{k_0}} (x, \fdot) \|\ci{\ell^q} |g(x)| dx \le
\sum_{k=1}^\infty \sum_{J\in \cG^*_k} \int_{J} \| (\wt\pi_b f)\ci{\cG(J)} (x, \fdot) \|\ci{\ell^q} |g(x)| dx
\end{equation}
Each integral in the sum (over $J\in \cG^*_k$) can be split
\[
\int_J \ldots = \int_{J\setminus G_{k+1}} \ldots + \int_{J\cap G_{k+1}} \ldots = A(J) + B(J).
\]

Let us estimate $A(J)$. For $J\in \cG^*_k$ let us denote $\wt J : = J \setminus G_{k+1}$. Note that the sets $\wt J$ are disjoint and $\cup_{k\ge1} \cup_{J\in \cG^*_k} \wt J = E_{k_0}$.

By the property \cond3 of generations $\cG^*_k$, the inequality $\La f\Ra\ci I\le 2^{\rf(J)+2}$ holds for $I\in\cG(J)$. Together with \eqref{4.1} this gives us the estimate
\begin{equation}
\label{4.2a}
\|(\wt \pi_b f)\ci{\cG(I)} \|\ci{L^p(\ell^q)} \le 2^{\rf(J)+2} |J|^{1/p},
\end{equation}
so
\begin{equation}
\label{4.3}
A(J) \le 4\cdot 2^{\rf(J)} |J|^{1/p} \| g \1\ci{\wt J}\|_{p'} \le
2^{1/p} 4\cdot 2^{\rf(J)} |\wt J|^{1/p} \| g \1\ci{\wt J}\|_{p'} ;
\end{equation}
the last inequality holds because by property \cond4 of generations $|J| \le 2 |\wt J |$ (recall that $\wt J = J\setminus G_{k+1}$).

We know that by definition of $\rf(J)$
\[
2^{\rf(J)} \le \min_{x\in J} M\ci\cL f(x),
\]
so using this inequality and disjointness of $\wt J$s we get
\begin{equation}
\label{4.5}
\sum_{k=1}^\infty \sum_{J\in \cG^*_k} 2^{\rf(J)\cdot  p} |J| \le 2\sum_{k=1}^\infty \sum_{J\in \cG^*_k} 2^{\rf(J)\cdot  p} |\wt J| \le 2  \int_\X [M\ci\cL f (x)]^p dx \le C \|f\|_p^p
\end{equation}
Again, since $\wt J$s are disjoint,
\[
\sum_{k=1}^\infty \sum_{J\in \cG^*_k}   \| g \1\ci{\wt J}\|_{p'}^{p'} \le \|g\|_{p'}^{p'}  .            
\]
So, applying H\"{o}lder inequality to \eqref{4.3} we get
\[
\sum_{k=1}^\infty \sum_{J\in \cG^*_k}  A(J) \le C \|f\|_p \|g\|_{p'} .
\]

To estimate the sum of $B(J)$s, let us notice that for $J\in \cG_k^*$ the function $(\wt \pi_b f)\ci{\cG(J)}$ is constant on intervals $I\in \cG^*_{k+1}$, $I\subset J$, so the integral $B(J)$ does not change if we replace $g$ there by the function $g\ci J$,
\[
g\ci J := \sum_{I\in \cG^*_{k+1}, I\subset J} \La g \Ra\ci I \1\ci I .
\]
Since
\[
\|g\ci J\|_{p'}^{p'} = \sum_{I\in \cG^*_{k+1}, I\subset J} | \La g \Ra\ci I |^{p'} |I|
\]
we can estimate using \eqref{4.2a}
\begin{equation}
\label{est-BJ}
B(J) \le 2^{\rf(J)+2} |J|^{1/p} \left( \sum_{I\in \cG^*_{k+1}, I\subset J} | \La g \Ra\ci I |^{p'} |I| \right)^{1/p'}
\end{equation}

The measure $\sum_{k=1}^\infty \sum_{J\in\cG^*_k} \1\ci J dx$ is clearly Carleson, i.e.~for any $I\in\cL$
\begin{equation}
\label{L-Carl-meas}
\sum_{k=1}^\infty \sum_{J\in\cG^*_k:J\subset I} \le C |I|.
\end{equation}
Indeed, since interval in each generation $\cG^*_k$ are disjoint, the desired inequality holds trivially if we consider only one generation in the sum, for example the generation $\cG^*_k$ with the smallest possible $k$, still intersecting $I$. By property \cond4 of generations, the contribution of each next generation is at most half of the previous, so summing geometric series we get \eqref{L-Carl-meas}.  

Therefore, by the Carleson Embedding theorem (cf.~Theorem \ref{t2.4} for $p=q$) we get that
\[
\sum_{k=1}^\infty \sum_{I\in \cG^*_{k+1}} | \La g \Ra\ci I |^{p'} |I| \le C \|g\|_{p'}^{p'}
\]
Therefore, summing \eqref{est-BJ} over all $J$, then applying H\"{o}lder inequality   and using \eqref{4.5}, we get
\[
\sum_{k=1}^\infty \sum_{J\in \cG^*_{k}}  B(J) \le C\|f\|_p \|g\|_{p'}.  
\]
Combining this with the estimate for the sum of $A(J)$s, we get from \eqref{pbf-g}
\[
\int_{E_{k_0}} \| (\wt\pi_b f)\ci{\cE_{k_0}} (x, \fdot) \|\ci{\ell^q} |g(x)| dx \le C\|f\|_p \|g\|_{p'}.
\]
Letting $k_0\to -\infty$
 concludes the proof. \hfill\qed

\subsection{Bounds for paraproducts }
We will need the following simple lemma. 
\begin{lm}
\label{Haar-norm}
Let $I$ be a disjoint union of sets $I_1$ and $I_2$, and let $h$ be a ``Haar function'', i.e. $h=\alpha_1\1\ci{I_1} + \alpha_2\1\ci{I_2}$ and  $\int_\X h dx =0$. Then, assuming without loss of generality that $|I_1|\le|I_2|$ we get that for $p\in[1, \infty)$
\[
\|h\|_p^p \le 2 \|h\1\ci{I_1}\|_p^p
\]
and that the inverse H\"{o}lder inequality holds
\[
\| h\|_p \|h\|_{p'} \le 2 \| h\|_2^2, \qquad 1/p+1/p'=1.
\]
\end{lm}

\begin{proof}[Proof of Lemma \ref{Haar-norm}]
Assume without loss of generality that $|I_1|\le |I_2|$. The condition $\int_\X h dx =0$ means that
\[
\alpha_1 |I_1| = - \alpha_2 |I_2|,  
\]
which immediately implies $|\alpha_2|\ge |\alpha_1|$. Then
\[
\int_\X |h|^p dx = |\alpha_1|^p |I_1| + |\alpha_2|^p |I_2| = |\alpha_1|^p |I_1| + |\alpha_2|^{p-1} |\alpha_1| |I_1|\le 2 |\alpha_1|^p |I_1|
\]
so
\[
\| h\|_p^p \le 2 \|h\1\ci{I}\|_p^p
\]
and similarly for $p'$.

Since for constant functions H\"{o}lder inequality becomes identity, using the above estimate we can write
\[
\| h\|_p \| h\|_{p'} \le 2 \| h\1\ci{I}\|_p \| h\1\ci{I}\|_{p'} = 2 \| h\1\ci{I}\|_2^2 \le 2 \| h\|_2^2
\]
\end{proof}

\begin{thm}
\label{t.para-bound}
Let $b=\{\Delta\ci I b\}\ci{I\in\cL}$ be a martingale difference sequence, and let $p\in [1, \infty)$, $q\in(1, \infty)$. Then
\begin{enumerate}
\item The paraproduct $\pi_b$ is a bounded operator 
from $\wt\cH^p$ to $\cH^p_q$ if and only if
\begin{equation}
\label{test1}
\sup_{I\in\cL} \fint_I \Bigl(\sum_{J\in\cL:J\subset I} | \Delta\ci J b (x) |^q    \Bigr)^{\frac1q p} dx = :K^p < \infty.
\end{equation}
Moreover
\[
K \le \|\pi_b\|\ci{\wt\cH^p\to\cH^p_q} \le CK,
\]
where $C=C(p,q)$.   
\item The paraproduct $\pi^{(*)}_b
 $ is a bounded operator 
 in $\wt\cH^p=\wt\cH^p_2$ 
 if and only if $b\in \BMOs = \BMOs_{2}$. Moreover
\[
\frac1C \|b\|\ci{\BMOs} \le \|\pi_b^{(*)}\|\ci{\wt\cH^p\to \wt\cH^p} \le C \|b\|\ci{\BMOs}
\]
where $C=C(p)$.   
\end{enumerate}
\end{thm}

\begin{rem}
For $q=2$ statement \cond1 of the theorem  describes the boundedness of the paraproduct $\pi_b$ in 
$\wt \cH^p$ (or equivalently, in $\wt H^p$). For $p\in (1, \infty)$ this is equivalent to the boundedness of $\pi_b$ in $L^p$. 

Note, that unlike the condition $b\in \BMOs$, which is necessary and sufficient for the boundedness of $\pi^{(*)}_b$ in all $\wt\cH^p$, the above condition \eqref{test1} (for $q=2$, for example) does depend on $p$.
\end{rem}

\begin{rem}
Note that the condition \eqref{test1} 
(for $p=q=2$) is weaker then the condition $b\in \BMO$. Since by Proposition \ref{p1.2}
\[
M_b = \pi_b + \pi^*_b +\Lambda_b + R_b,
\]
and $\lambda_b$, $R_b$ commute with all martingale multipliers, the above theorem implies, in particular, that unlike the homogeneous case, it is impossible in general to characterize $b\in\BMO$ via boundedness of the commutators of $M_b$ with \emph{martingale multipliers}.
\end{rem}

\begin{proof}[Proof of Theorem \ref{t.para-bound}]
The statement \cond1 is easy. The ``only if'' part and the estimate $K\le \|\pi_b\|$ follow from testing the boundedness of $\pi_b$ on   functions $\1\ci I$, $I\in \cL$.   The ``if'' part with the estimate $\|\pi_b\|\le CK$ follow from Theorem \ref{t.para-embed} above.

Let us prove statement \cond2.
Notice that by Proposition \ref{p1.2}
\[
\pi_b^{(*)} = \pi_b^* + \Lambda_b^1.
\]
If $b\in\BMOs$, we know that for any $p'\in(1, \infty)$
\[
\sup_{I\in\cL} \fint_I \Bigl(\sum_{J\in\cL:J\subset I} | \Delta\ci J b (x) |^{2}    \Bigr)^{\frac1{2}{ p'} } dx \le C \|b\|_{\BMOs}^{p'}.
\]
Taking $p'$ to be the dual exponent to $p$, $1/p+1p'=1$, we get that
 by \cond1 $\pi_b$ is bounded in $\wt\cH^{p'}$, so  by duality $\pi_b^*$ is bounded in $\wt\cH^p$.

Since by Proposition \ref{p1.2}
\begin{equation}
\label{La_b1-1}
\Lambda_b^1 f = \sum_{I\in\cL} \Delta\ci I \bigl[ (\Delta\ci I b)(\Delta\ci I f ) \bigr]
=  \sum_{I\in\cL}   (\Delta\ci I b)(\Delta\ci I f ) - \sum_{I\in\cL} \E\ci I( (\Delta\ci I b)(\Delta\ci I f ) )  
\end{equation}
and by the definition of $\BMOs$ we have $\|\Delta\ci I b\|_\infty \le \|b\|_{\BMOs}$, we can conclude that $\Lambda_b^1$ is bounded in $\wt\cH^p$. Indeed, since $\|\Delta\ci I b\|_\infty \le \|b\|_{\BMOs}$
\[
\int_\X \sum_{I\in\cL} \bigl| (\Delta\ci I b)(\Delta\ci I f ) \bigr|^p dx \le
\|b\|_{\BMOs}^p \int_\X  \sum_{I\in\cL} \bigl| \Delta\ci I f  \bigr|^p dx
= \|b\|_{\BMOs}^p  \|f\|_{\wt\cH^p}^p  .  
\]
By Fefferman--Stein maximal theorem we get from this inequality
\[
\int_\X \sum_{I\in\cL} \bigl| \E\ci I( (\Delta\ci I b)(\Delta\ci I f ) ) \bigr|^p dx \le
C \|b\|_{\BMOs}^p  \|f\|_{\wt\cH^p}^p .
\]
So both sums in \eqref{La_b1-1} can be estimated and we get that $\Lambda_b^1$ is bounded in $\wt\cH^p$.

Assume now that $\pi_b^{(*)}$ is bounded in $\wt\cH^p$, so
\[
\bigl(\pi_b^{(*)} \bigl)^* = \pi_b + (\Lambda_b^1)^* = \pi_b + \Lambda_b^1
\]
is bounded in $\wt\cH^{p'}$.   Testing this operator on functions $\1\ci I$ and counting in the result only martingale differences  with $J\subset I$, we get
\begin{equation}
\label{4.17}
\sup_{I\in\cL} \fint_I \Bigl(\sum_{J\in\cL:J\subset I} | \Delta\ci J b (x) |^{2}    \Bigr)^{\frac1{2}{ p'} } dx \le K^{p'} < \infty, \qquad K =  \bigr \|\bigl( \pi_b^{(*)}\bigr)^* \bigr\|_{\cH^p \shto \cH^p}
\end{equation}
By \cond1 this means that $\pi_b$ is bounded in $\cH^{p'}$ (with the norm at most $CK$), and so $\Lambda_b^1$ is also bounded in $\cH^{p'}$ with the norm at most $C_1 K$. By duality, $\Lambda_b^1$ is   bounded in $\cH^p$ (with the same norm).

Estimate \eqref{4.17} also implies that $\|\Delta\ci I b\|_{p'} \le K |I|^{1/p'}$.

To prove that $b\in\BMOs_{q'}$ it only remains to show that for all $I\in\cL$
\[
 \|\Delta\ci I b \|_\infty \le CK.
\]

Assume that $\|\Delta\ci I b\|_\infty \ge 2^{1/p} K$, because otherwise we already have the desired estimate. Let $J\in \chld(I)$ be an interval where $\|\Delta\ci I b\|_\infty$ is attained. Then
\begin{align*}
2 K^{p} |J| \le \|\Delta\ci I b\|_\infty^p |J| & \le \| \Delta\ci I b\|_{p}^{p} K^{p} |I|,
\end{align*}
so  $|J|  \le |I|/2$.

Define a test function $h$
\[
h:= \1\ci J - \alpha \1\ci{I\setminus J} , \qquad \alpha = |J|/(|I| -|J|) \le 1,\]
so $\int_\X h dx =0$.   Since $|J| \le | I\setminus J|$, Lemma \ref{Haar-norm} implies that
\[
\| h\|_{p}^{p} \le 2 \| \1\ci J \|_{p}^{p} = 2|J| .  
\]

For out test function $\Delta\ci I h =h$ is the only non-zero martingale difference, so it follows from \eqref{La_b1-1} that  
\[
\| h \Delta\ci I b\|_{p} \le \| \Lambda^1_b h \|_{p} + \| \E\ci I (h \Delta\ci I b) \|_{p} .
\]
We can estimate
\[
\| h \Delta\ci I b\|_{p} \ge \| 1\ci{\! J}\, h \Delta\ci I b\|_{p} = |J|^{1/p} \|\Delta\ci I b\|_\infty .
\]

On the other hand,
\begin{align*}
\| \Lambda^1_b h \|_{p}  & \le C_1 K \|h\|_{p} = C_1 K (2|J|)^{1/p} \\
\intertext{and}
\|  \E\ci I (h \Delta\ci I b) \|_{p} & \le  |I|^{1/p}  \fint_I |h\Delta\ci I b| dx \\
&\le
|I|^{1/p -1}  \|\Delta\ci I\|_{p'} \|h\|_p \\
& \le
|I|^{1/p -1}  K |I|^{1/p'} (2|J|)^{1/p} = 2^{1/p} K |J|^{1/p}
\end{align*}
Combining all together we get that
\[
|J|^{1/p} \|\Delta\ci I b\|_\infty \le 2^{1/p} C_1 K |J|^{1/p} + 2^{1/p}  K |J|^{1/p},
\]
so $\|\Delta\ci I b\|_\infty \le C K$, $C=2^{1/p}(C_1+1)$.
\end{proof}

\section{Boundedness of commutators in \texorpdfstring{$L^p$}{L**p}}
\subsection{Sufficiency}
We start with a simple proposition.

\begin{prop}
\label{p.BMO-comm}
Let $p\in(1, \infty)$, and let $T$ be a bounded in $L^p$ (equivalently in $\wt H^p_2$) martingale transform. Let $b$ be a locally integrable function.  

If the formal sum $b_0:= \sum_{I\in\cL} \Delta\ci I b$ is in $\BMOs$, the the commutator $[M_b, T] = M_b T - TM_b$ is bounded in $L^p$ (equivalently in $\wt H^p$). Moreover, 
\[
\| [M_b, T]\|\ci{L^p\shto L^p} \le C \|T\|\ci{L^p\shto L^p} \|b_0\|\ci{\BMOs}, 
\]
 where $C=C(p)$. 
\end{prop}

\begin{rem*}
Note, that the case $p=1$ is not included here. While the condition $b_0\in \BMOs$ is necessary and sufficient for the boundedness of the paraproduct $\pi_b$ in $\wt H^1$, this condition is not sufficient for the boundedness of the adjoint $\pi_b^*$ there, even in the simplest case of the standard dyadic grid. 

This can be easily seen by going to the dual space and noticing that the condition $b\in \BMO$ is not sufficient for the boundedness of the paraproduct $\pi_b$ in $\BMO$ (we are considering the simplest case of the standard dyadic grid on $\R$ here, so all BMO spaces coincide). Since the condition $f\in \BMO$ does not imply any bounds on the averages $\La f\Ra\ci I$, on can take an unbounded function $f\in\BMO$ (so the averages $\La f\Ra\ci I$ are not uniformly bounded) and easily construct a function $b\in\BMO$ such that $\pi_b f\notin \BMO$. 
\end{rem*}

\begin{proof}[Proof of Proposition \ref{p.BMO-comm}]  
By Proposition \ref{p1.2}
\[
M_b    = \pi_b^{(*)} + \Lambda_b^0 + \pi_b + R_b.
\]
Operator $\Lambda_b^0$ commutes with all martingale transforms, so we can exclude it from the commutator. Since $T R_b = R_b T =0$, we can exclude $R_b$ as well, so 
\[
[M_b ,  T] = [\pi_b + \pi^{(*)}_b , T]. 
\] 
Therefore, if $\pi_b$ and $\pi_b^{(*)}$ are bounded, the commutator is bounded as well. But according to Theorem \ref{t.para-bound}, the condition $b\in \BMO$ implies the boundedness of both paraproducts $\pi_b$ and $\pi^{(*)}_b$ (for $q=2$ condition \cond1 of Theorem \ref{t.para-bound} follows from condition \cond2 there).
\end{proof}

It will be shown later that in the case when $\Aif \cap \cL =\varnothing$ and the martingale transform $T$ has the right ``mixing'' properties, the condition  $\sum_{I\in\cL} \Delta\ci I b \in \BMOs$ is also necessary for the boundedness of the commutator.  
If $\Aif \cap \cL \ne\varnothing$,  the sufficient condition $\sum_{I\in\cL} \Delta\ci I b \in \BMOs$ can be relaxed a little. As it will be shown below in Section \ref{s-nec-BMO} this relaxed condition is also necessary (again if the martingale transform $T$ has the right ``mixing'' properties).

\subsection{Necessity}
\label{s-nec-BMO} We want to state and prove an inverse (at least partial) to the above Proposition \ref{p.BMO-comm}. Of course, to prove such a theorem one needs to make some additional assumptions about the martingale transform $T$ (for example, identity is a martingale transform, and it commutes with everything).  

\begin{df}
\label{df.non-deg}
Let $T$ be a martingale transform. Following S.~Janson \cite{Janson-CommMartBMO_1981}, we say that an interval $I\in\cL$ with parent $I'$ is $(p, \e, K)$ non-degenerate for $T$ if  there exists $h =h\ci{I'}\in D\ci{I'} = \Delta\ci{I'} L^2$,
such that
\begin{enumerate}
\item $\|h \|_p=1$,

\item $h|_{I}=0$ ,
\item $\| \1\ci I T\ci{I'} h \|_p \ge \e$,
\item  $\|h\|_\infty \le K |I'|^{-1/p} $ if $I$ is ``small'', namely if $|I|< |I'|/K$.
\end{enumerate}
The last condition \cond4 means that for ``small'' intervals $I$ the function $h$ has to be ``spread'' on the interval $I'$.

If we skip condition \cond4, we get the definition of $(p, \e)$ non-degenerate interval.  

We say that the martingale transform $T$ is weakly $(p, \e, K)$ mixing if each interval $I$ with a parent is either $(p, \e, K)$ non-degenerate for $T$ or $(p', \e, K)$ non-degenerate for the adjoint $T^*$.

We say that the martingale transform $T$ is strongly $(p, \e, K)$ mixing if each interval $I$ with a parent is  $(p, \e, K)$ non-degenerate.

Using the notion of $(p,\e)$ non-degenerate intervals, one can define weakly and strongly $(p, \e)$ mixing martingale transforms.

\end{df}

\begin{rem*}
The above definition of weakly $(p, \e)$ mixing martingale transform is essentially a restatement (and a generalization) of the  definition of a non-degenerate transform from \cite{Janson-CommMartBMO_1981}. It was given there for case of the uniform $r$-adic lattice, with all operators $T\ci I$ being equal (after canonical identification of all subspaces $D\ci I$).

For the case considered in \cite{Janson-CommMartBMO_1981}, our definition coincides with one given there. The easiest way to see this equivalence is to look directly at the proof of Theorem 2 in \cite{Janson-CommMartBMO_1981} (at least that was the easiest way for me).

Note also, that for homogeneous lattices the norms $\|f\|_p |I|^{-1/p}$ on $D\ci I$ are all equivalent. That means that any (weakly or strongly) $(p, \e)$ mixing martingale transform is also $(p, \e, K)$ mixing (resp.~weakly or strongly) with appropriate $K$.   It also mean that any $(p, \e)$ mixing martingale transform is also $(r, \e')$ mixing with appropriate $\e'$.
\end{rem*}

Recall that we defined the formal sum $b_0=\sum_{I\in\cL}\Delta\ci I b$. Define also the formal sum 
\begin{equation}
\label{wt-b_0}
\wt b_0:= \sum_{I\in\cL\setminus \Aif} \Delta\ci I b
\end{equation}
(note that $b_0=\wt b_0$ if $\Aif\cap \cL=\varnothing$). 

\begin{thm}
\label{t.comm-bmo}
Let $p\in (1, \infty)$ and let $T$ be  a strongly $(p, \e, K)$ mixing   martingale transform, such that its blocks $T\ci I$ are uniformly bounded in $L^p$.%
\footnote{Note, that for $p\ne 2$ this condition is weaker than boundedness of $T$ in $L^p$.}

If the commutator $[T, M_b]$ is bounded in $L^p$,  then $\wt b_0\in \BMOs=\BMOs_2$;

Moreover, for $p=2$ it is sufficient to assume that $T$ weakly $(2, \e, K)$ mixing   martingale transform. 

Finally, the norm $\| \wt b_0\|\ci{\BMOs}$ can be estimated by a constant depending on $p$, $\|[M_b, T]\|\ci{L^p\shto L^p}$, $\sup_{I\in\cL}\|T\ci I\|$ and $\e$, $K$ from Definition \ref{df.non-deg}.
\end{thm}

\begin{prop}
\label{p.comm-bmo-}
Let $p\in (1, \infty)$ and let  $T$ be a (possibly unbounded)  strongly $(p, \e)$ mixing   martingale transform (weakly $(p, \e)$ mixing for $p=2$).

If the commutator $[M_b, T]$ is bounded in ${\cH^p_2}$, then for any interval $I\in\cL\setminus \Aif$  a uniform estimate,
\begin{equation}
\label{comm-bmo-}
\fint_I \biggl(\sum_{J\in\cL:J\subset I} |\Delta\ci J b|^2  \biggr)^{p/2} dx \le C <\infty, \qquad C^{1/p}=C_1 \|[M_b, T]\|/\e.   
\end{equation}
where $C_1=C_1(p)$, holds.
\end{prop}

\begin{proof}[Proof of Proposition \ref{p.comm-bmo-}]

The  proof directly follows \cite{Janson-CommMartBMO_1981}. For an interval $I$, let $I'$ be its parent, so $I\in\chld(I')$. We know that  that $I$ is $(p, \e)$ non-degenerate. Let $h=h\ci{I'}\in D\ci{I'}$ be the function from Definition \ref{df.non-deg} such that $\|h \|_p=1$, $h|_{I}=0$ and $\|\1\ci I T\ci{I'} h\|_p\ge \e$.

Note that
$\|h\|_p = \|h\|_{{\cH^p_2}}$.

Recall that the function $Th=T\ci{I'} h$ is constant on $I$, and let $c$ be its value there. The inequality $\|\1\ci I T\ci{I'} h\|_p\ge \e$ means that $|c|\ge \e |I|^{-1/p}$.

We get that on $I$
\[
M_b T h  = c b.
\]
On the other hand, $bh=0$ on $I$, so $(Th)|_{I}$ is a constant, so for $J\subset I$
\[
\Delta\ci J \Bigl((M_b T - TM_b)h\Bigr)  = c\Delta\ci J h.
\]
The fact that $M_b T - TM_b$ is bounded in $\crc{\cH^p_2}$ implies that
\[
|c | \cdot \biggl\| \biggl( \sum_{J\in \cL J\subset I}\Delta\ci J \biggr)^{1/2} \biggr\|_p\le C, \qquad C = \|[M_b, T]\|,
\]
so taking into account that $|c|\ge \e |I|^{-1/p}$ we get the conclusion of the proposition.

For $p=2$, we can assume that $T$ is weakly  $(p, \e)$ mixing, because if $I$ is $(p, \e)$ non-degenerate for $T'$, we can consider the adjoint of the commutator, to get the same conclusion.  
This would not work for $p\ne 2$, because in this case we get the estimate with the exponent $p'$ instead of $p$.
\end{proof}


\subsubsection{Proof of Theorem \ref{t.comm-bmo}}


To prove the theorem we need to show that $\|\Delta\ci{I} b\|_\infty$ are uniformly bounded for all $I\in\cL\setminus \Aif$. 

Consider an interval (let us call it $I'$) belonging to $\cL\setminus\Aif$. Notice that inequality   
 \eqref{comm-bmo-} implies that $\|\Delta\ci{ I'} b\|_p\le C^{1/p} |I'|^{1/p}<\infty$.
Assume that $M:=\|\Delta\ci{I'} b\|_\infty$ is attained on $I\in\chld(I')$.


We can assume that $|I|<|I'|/K$, because otherwise
\[
\| \1\ci{I} \Delta\ci{I'} b \|_\infty^p  =  |I|^{-1} \|\1\ci{I} \Delta\ci{I'} b \|_p^p
\le
|I|^{-1} \| \Delta\ci{I'} b \|_p^p
\le |I|^{-1} C |I'| \le KC .   
\]

Define
\[
g= \1\ci I - \gamma \1\ci{I'\setminus I},
\]
where the constant $\gamma$ is chosen so $\int_\X g dx =0$. Let $E\subset D_{I'}$ be the annihilator of $g$ in $D_{I'}$
\[
E = \Bigl\{ f\in D_{I'}: \int_\X gf dx = 0 \Bigr\} .
\]
Note that $E$ consist of all functions $f\in D_{I'}$ supported outside of $I$. Indeed, any such function annihilates $g$, and counting dimensions, we can conclude that we got all the functions in the annihilator.

Such structure of $E$ implies that
\[
\int_\X f\overline g dx = 0 \qquad \forall f\in E,
\]
so $E$ is the orthogonal complement of $g$ in $D_{I'}$. Therefore, $D_{I'}$ can be decomposed into the direct sum of $\spn\{g\}$ and $E$.

We can decompose
\begin{equation}
\label{T_I'h}
(\Delta\ci{I'} b) T\ci{I'} h = \alpha g + f + \E\ci{I'} [(\Delta\ci{I'} b) (T\ci{I'} h)], \qquad f\in E.
\end{equation}
By the assumption \cond3 about $h$, $|T\ci{I'} h|\ge \e |I|^{-1/p}$ on $I$. Therefore, since $f\bigm|_{I}=0$ and $g\bigm|_{I} =1$,  we get from \eqref{T_I'h} by restricting it to $I$ and comparing $L^p$ norms (divided by $|I|^{1/p}$), that
\begin{align}
\label{5.3}
 \e |I|^{-1/p} \| \Delta\ci{I'} b\|_\infty  & \le |\alpha | + \| \E\ci{I'} [(\Delta\ci{I'} b) (T\ci{I'} h)]\|_\infty &&
\\
\notag
 & \le |\alpha | + |I'|^{-1} \| \Delta\ci{I'} b\|_{p'} \|T\ci{I'}\| \|h\|_p   &&
 \\  \notag
 &\le |\alpha | + C |I'|^{-1} \| \Delta\ci{I'} b\|_{p'} . &&  
\end{align}
So, to estimate $\| \Delta\ci{I'} b\|_\infty$ we need to estimate both terms in the right side of \eqref{5.3}.

We get the bound on $|\alpha|$ from the boundedness of the commutator. Namely,
since
\[
b\bigm|_{I'} = \E\ci{I'} b +\Delta\ci{I'} b + \sum_{J\in L:J\subsetneqq I'} \Delta\ci{J'} b =: \E\ci{I'} b +\Delta\ci{I'} b + b^{I'}
\]
and $b^{I'} \perp D\ci{I'}$, $b^{I'} D\ci{I'} \perp D\ci{ I'}$, we can write
\begin{alignat*}{2}
\La M_b T h, g \Ra & = \La M_b T\ci{I'} h, g \Ra & & = \La (\Delta\ci{I'} b) T\ci{I'} h, g \Ra + \La (\E\ci{I'} b) T\ci{I'} h, g \Ra, \\
\La T M_b  h, g \Ra & = \La T\ci{I'} M_b  h, g \Ra   &  &= \La T\ci{I'} (\Delta\ci{I'} b)  h, g \Ra + \La T\ci{I'} (\E\ci{I'} b)  h, g \Ra.
\end{alignat*}
Here we slightly abusing the notation by treating $T\ci{I'}$ as the operator on all $L^p$, i.e.~as a martingale transform whose only non-zero block is $T\ci{I'}$ (we need to do that because $(\Delta\ci{I'} b)  h$ does not generally belongs to $D\ci{I'}$). In this context $T\ci{I'} (\Delta\ci{I'} b)  h = 
T\ci{I'} \bigl[ (\Delta\ci{I'} b)  h - \E\ci{I'}( (\Delta\ci{I'} b)  h ) \bigr]$, where $T\ci{I'}$ in the right side  can be treated as a block acting in $D\ci{I'}$. 
 
Using the fact that $(\E\ci{I'} b) T\ci{I'} h = T\ci{I'} (\E\ci{I'} b)  h$ we conclude, again abusing the notation as above,  that for the commutator $[M_b, T] = M_b T - TM_b$
\begin{equation}
\label{5.4}
\La [M_b, T] h, g \Ra = \La (\Delta\ci{I'} b) T\ci{I'} h, g \Ra - \La T\ci{I'} (\Delta\ci{I'} b)  h, g \Ra.
\end{equation}
We get from \eqref{T_I'h} that
\begin{equation}
\label{5.5}
\left| \La (\Delta\ci{I'} b) T\ci{I'} h, g \Ra \right| = |\alpha| \cdot \|g\|_2^2 \ge |\alpha| \cdot |I|.
\end{equation}

By Lemma \ref{Haar-norm} $\| g\|_{p'} \le 2^{1/p'} \|\1\ci I\|_{p'} = 2^{1/p'} |I|^{1/p'}$. Using this estimate and
 the assumption $\|h\|_\infty\le K |I'|^{-1/p}$, we get
\begin{align*}
\left|    \La T\ci{I'} (\Delta\ci{I'} b)  h, g \Ra \right| & \le \|T\ci{I'}\| \cdot \|h\|_\infty \|\Delta\ci{I'} b\|_p \|g\|_{p'} \\
& \le C K |I'|^{-1/p} C |I'|^{1/p} 2^{1/p} |I|^{1/p'} \le C |I|^{1/p'} .
\end{align*}
Using the above estimate together with the estimate
\[
\left| \La [M_b, T] h, g \Ra \right| \le C \|h\|_p \|g\|_{p'} \le C \cdot 1\cdot |I|^{1/p'}
\]
we get from \eqref{5.4} and \eqref{5.5} that
\begin{align*}
|\alpha| \cdot |I| &\le  \left| \La (\Delta\ci{I'} b) T\ci{I'} h, g \Ra \right| \\
& \le \bigl| \La [M_b, T] h, g \Ra \bigr| + \bigl|    \La T\ci{I'} (\Delta\ci{I'} b)  h, g \Ra \bigr| 
\\
&\le C |I|^{1/p'} + C |I|^{1/p'} = C |I|^{1/p'} , 
\end{align*}
so
\[
|\alpha | \le C |I|^{-1/p}.
\]
Combining the last inequality with \eqref{5.3} we get
\begin{align}
\label{5.6}
\|\Delta\ci{I'} b\|_\infty & \le \frac{|I|^{1/p}}{\e}  |\alpha|  + C \frac{|I|^{1/p}}{\e}\|\Delta\ci{I'} b\|_{p'} |I'|^{-1}
\\
\notag
& \le
C + C |I|^{1/p} |I'|^{-1} \|\Delta\ci{I'} b\|_{p'}
\end{align}

If $p'\le p$, H\"{o}lder inequality implies that
\[
|I'|^{-1/p'} \| \Delta\ci{I'} b\|_{p'} \le |I'|^{-1/p} \| \Delta\ci{I'} b\|_{p} \le C
\]
so
\[
\|\Delta\ci{I'} b\|_\infty \le C + C |I|^{1/p} |I'|^{-1/p} \le C'.
\]
If $p'>p$, Lemma \ref{l.5.5} below implies that
\[
\|\Delta\ci{I'} b\|_{p'} \le  \|\Delta\ci{I'} b\|_p^{p/p'}  \|\Delta\ci{I'} b\|_\infty^{1-p/p'}
\]
and we get from \eqref{5.6}
\begin{align}
\label{5.7}
\|\Delta\ci{I'} b\|_\infty   
& \le C + C |I|^{1/p} |I'|^{-1} \|\Delta\ci{I'} b\|_p^{p/p'}  \|\Delta\ci{I'} b\|_\infty^{1-p/p'}
\\  \notag
& \le C + C |I|^{1/p} |I'|^{-1/p} \|\Delta\ci{I'} b\|_\infty^{1-p/p'},
\end{align}
the last inequality being true because
\[
|I'|^{-1/p'} \|\Delta\ci{I'} b\|_p^{p/p'} =
\left( |I'|^{-1/p} \|\Delta\ci{I'} b\|_p \right)^{p/p'} \le C^{p/p'} \le C'.
\]
Since $|I|\le|I'|$, \eqref{5.7} implies
\[
\|\Delta\ci{I'} b\|_\infty  \le C + C \|\Delta\ci{I'} b\|_\infty^{1-p/p'},
\]
which gives us a bound $\|\Delta\ci{I'} b\|_\infty  \le C'$.
\hfill\qed

\begin{lm}
\label{l.5.5}
Let $f$ be a bounded measurable function on a measure space $\X$. Then for any $q>p$
\[
\|f\|_q \le \|f\|_p^{p/q} \|f\|_\infty^{1-p/q}
\]
\end{lm}
\begin{proof}
\[
\|f\|_q^q = \int_\X |f|^q d\mu = \int_\X |f|^p |f|^{q-p} d\mu \le
\|f\|_\infty^{q-p} \int_\X |f|^p d\mu  = \|f\|_\infty^{q-p} \| f\|_p^p
\]
and raising this inequality to the power $1/q$ we get the conclusion of the lemma.
\end{proof}
 
\subsection{Relaxing sufficient condition} If $\cL\cap \Aif = \varnothing$, we have $b_0=\wt b_0$, so 
$b_0\in\BMOs$ 
is a necessary and sufficient condition for the boundedness of the commutator $[M_b, T]$ (provided that $T$ satisfies assumptions of Theorem \ref{t.comm-bmo}). 

If $\cL\cap \Aif \ne \varnothing$ there is a gap between necessary and sufficient conditions. Notice, that the situation $\cL\cap \Aif \ne \varnothing$ is not an exotic one. For example, it happens in the classical martingale situation, which in our notation mean that $\cL_k = \cL_0 =\X$ for all $k<0$, $|\X|=1$. 

To bridge the gap between necessary and sufficient conditions in the case $\cL\cap \Aif \ne \varnothing$, we can relax sufficient conditions in Proposition \ref{p.BMO-comm}.

\begin{prop}
\label{p.BMO-comm-2}
Let $b$ be a locally integrable function, and let $T$ be a bounded in $L^p$ martingale transform.  Assume that
\begin{enumerate}
\item $\wt b_0\in \BMOs$, where $\wt b_0$ is defined by \eqref{wt-b_0}; 
\item For any $I\in\cL\cap \Aif$
\[
\| T\ci I \Delta\ci I b \|_p \le C_1 \|\1\ci I\|_p =C_1 |I|^{1/p}, \qquad 
\| T\ci I^* \Delta\ci I b \|_{p'} \le C_1 \|\1\ci I\|_{p'} =C_1 |I|^{1/p'}  ;
\]
\item For any $I\in\cL\cap \Aif$
\[
\left\| \left[ (\Lambda_b^1)\ci I, T\ci I \right] \right\|_{L^p\shto L^p} \le C_2 <\infty;
\]
here $(\Lambda_b^1)\ci I$ is the restriction of $\Lambda_b^1$ on $D\ci I$. 
\end{enumerate}
Then the commutator $[M_b, T]$ is bounded in $L^p$, and 
\[
\left\|  [M_b, T]  \right\|_{L^p\shto L^p} \le C \left( \|T\|_{L^p\shto L^p} \|\wt b_0\|\ci{\BMOs} + C_1 +C_2 \right), 
\] 
where $C=C(p)$ and $C_1$ $C_2$ are the constants from \textup{\cond2}, \textup{\cond3}. 
\end{prop}

The proof of the theorem is obvious, since for any $I\in \cL\cap \Aif$ the conditions \cond2, \cond3  are necessary and sufficient for the boundedness of the commutator $[M_{b_0-\wt b_0}, T]$ in $L^p$.  The necessity here is quite easy: condition \cond2 is obtained by testing the commutator  $[M_{b_0-\wt b_0}, T]$ and its adjoint on the function $\1\ci I$. To get the condition \cond3 one needs to restrict everything to the subspace $D\ci I$. 

\begin{rem}
As it follows from the above discussion, it $T$ satisfies assumptions of Theorem \ref{t.comm-bmo}, then  conditions \cond1--\cond3 of Proposition \ref{p.BMO-comm-2} are necessary and sufficient for the boundedness of the commutator $[M_b, T]$ in $L^p$. 
\end{rem}

\subsection{Some examples and counterexamples} In this subsection we present examples which will show us that 
\begin{enumerate}
\item Boundedness of the commutator $[M_b, T]$ does not imply any bounds on $\Delta\ci I b$ for $I\in \cL\cap \Aif$;
\item If the martingale transform $T$ is only strongly $(p, \e)$ mixing (not strongly $(p, \e, K)$ mixing), then the boundedness of the commutator $[M_b, T]$ does not imply any bounds on $\|\Delta\ci I b\|_\infty$, $I\in \cL$.  That means that the new condition \cond4 in  Definition \ref{df.non-deg} is essential and cannot be skipped. 
\end{enumerate}

The main building block of our construction will be as follows. Let an interval $I$ be divided into 2 subintervals $I^{1,2}$, $|I^1|/|I^2| = \delta>0$. 
Divide $I^1$ into 4 equal intervals $I_k$, $1\le k \le4$ and $I^2$ into $4$ equal intervals $I_k$, $5\le k \le 8$. 

The intervals $I_k$ will be the children of $I$. Define the ``Haar functions'' $h^k = h^k_I\in D\ci I$
\[
h^k :=\1\ci{I_{2k}} - \1\ci{I_{2k-1}}, \qquad 1\le k \le 4. 
\]
Note, that the functions $h^k$ do not span the  martingale difference subspace $D\ci I$. Define also a ``Haar function'' $h=h\ci I\in D\ci I$, $h= \1\ci{I^1} - \delta \1\ci{I^2}$.

On $D\ci I$ define a block $T\ci I$, 
\begin{alignat*}{2}
T\ci I h^1 & = h^2, \qquad & T\ci I h^2 &= h^1, \\
T\ci I h^3 & = h^4, & T\ci I h^4 &= h^3, \qquad T\ci I\Bigm|_{\spn\{h^k: 1\le k \le 4\}^\perp } =0. 
\end{alignat*}

If $\Delta\ci I b = \alpha h\ci I$, then the block $(\Lambda_b^1)\ci I$ of $\Lambda_b^1$ commutes with $T\ci I$. This together with the fact  that $T\ci I h\ci I =0$ implies that if $I\in \cL\cap\Aif$ and the block $T\ci I$ of a martingale transform $T$ is as described above, then multiplication operator $M_{h_I}$ commutes with $T$ 

So, if we add to $b$ any multiple of $h\ci I$, we will not be able to detect it by looking at the commutator $[M_b, T]$, which gives a example for the statement \cond1 above. 

To give an example to statement \cond2, 
take a finite  interval $I_0 =:\X$, divide it into $8$ subintervals, as it was described above (with $\delta=\delta_1$) to get the ``children'' of $I_0$, then divide each child into $8$ parts, and so on. We assume that on each step we take $\delta=\delta_n$, $\delta_n\to 0$ as $n\to \infty$. That will be our lattice $\cL$.

Let $T$ be a martingale transform on $\cL$, where each block $T\ci I$ is as described above. Notice, that $T$ is strongly $(p, \e)$ mixing (but not strongly $(p, \e, K)$ mixing). Notice also, that clearly $T$ is bounded in $L^2$.

Take an interval  $I\in \cL\setminus\Aif =\cL\setminus \{I_0\}$. 

Take $p=2$ and define $\wt h =\wt h\ci I = \delta^{-1/2}  h$, where $h=h\ci I$ is the ``Haar function'' defined above, $h= \1\ci{I^1} - \delta \1\ci{I^2}$. 

By Lemma \ref{Haar-norm},  $\| \wt h\|_2 \le 2^{1/2} |I|^{1/2}$. On the other hand, $\|\tilde h\|_\infty =\delta^{-1/2}$, so we can pick $I$ such that $\|\wt h\|_\infty$ is  as large as we want.  

Note that for $b=\wt h$, the martingale transform $T$ commutes with $\Lambda_b^1$ (and so with $\Lambda^b$), so it is easy to check that the paraproducts $\pi_b$, $\pi^*_b$ and so the commutator $[M_b, T]$ are bounded. However, as we discussed above, $\|b\|_\infty = \delta^{-1/2}$. 

So, if we consider a collection $\cC$ of disjoint intervals in $\cL\setminus \Aif$ with $\delta \to 0$, and define 
\[
b = \sum_{I\in\cC} \wt h\ci I
\]
then the commutator $[M_b, T]$ is bounded. That can be seen, for example, by noticing that $\Lambda^1_b$ commutes with $T$ (one needs to treat each block separately, which reduces it to the case $b=\wt h\ci I$), and the paraproducts $\pi_b$ and $\pi_b^*$ are ``direct sums'' of the paraproducts with $b=\wt h\ci I$, treated above.  

So we constructed an example of $b$ and a strongly $(p, \e)$ mixing martingale transform $T$ such that the commutator $[M_b, T]$ is bounded in $L^2$, but $\sup_{I\in \cL} \|\Delta\ci I\|_\infty = \infty$. 

An easy modification allows also to get an example for $L^p$.

\def\cprime{$'$}
\providecommand{\bysame}{\leavevmode\hbox to3em{\hrulefill}\thinspace}

\end{document}